\documentclass[reqno,11pt]{amsart}
\usepackage{amsmath,amsfonts,amssymb,amsxtra,latexsym,amscd,enumerate,amsthm,verbatim}

\usepackage[margin=1.40in]{geometry}
\numberwithin{equation}{section}

\newcommand{\R}{\mathbb{R}}

\newcommand{\T}{\mathbb{T}}

\newcommand{\Z}{\mathbb{Z}}

\numberwithin{equation}{section} 

\newtheorem{theorem}{Theorem}[section]
\newtheorem{lemma}[theorem]{Lemma}
\newtheorem{proposition}[theorem]{Proposition}
\newtheorem{corollary}[theorem]{Corollary}
\newtheorem{definition}[theorem]{Definition}
\newtheorem{remark}[theorem]{Remark}

\newtheorem{assumption}[theorem]{Assumption}

\begin{document}

\title[On the stability of non-monotonic shear flows]{On the stability of shear flows in bounded channels, II: non-monotonic shear flows}

\author{Alexandru D. Ionescu}
\address{Princeton University}
\email{aionescu@math.princeton.edu}

\author{Sameer Iyer}
\address{University of California, Davis}
\email{sameer@math.ucdavis.edu}

\author{Hao Jia}
\address{University of Minnesota}
\email{jia@umn.edu}



\thanks{The first author is supported in part by NSF grant DMS-2007008.  The second author is supported in part by NSF grant DMS-2306528. The third author is supported in part by NSF grant DMS-1945179. }

\begin{abstract}
{\small}
We give a proof of linear inviscid damping and vorticity depletion for non-monotonic shear flows with one critical point in a bounded periodic channel. In particular, we obtain quantitative depletion rates for the vorticity function without any symmetry assumptions. 
\end{abstract}

\maketitle

\begin{center}
\large\textit{Dedicated to Carlos Kenig, on the occasion of his 70th birthday.}    
\end{center}

\bigskip


\setcounter{tocdepth}{1}

\tableofcontents


\section{Introduction}
The study of stability problems in mathematical analysis of fluid dynamics has a long and distinguished history, dating back to the work of Kelvin \cite{Kelvin}, Orr \cite{Orr} and Rayleigh \cite{Rayleigh} among many others, and continuing to the present day. Hydrodynamical stability problems can be considered in both two and three dimensions. In this paper we work with two dimensional inviscid flows. 

For the Euler equations, there is significant recent progress on the asymptotic stability of monotonic shear flows and vortices, assuming spectral stability, see for example \cite{Grenier,dongyi,Zillinger1,Zillinger2, JiaL,JiaG,Bed2,JiaVortex,Xiaoliu,Stepin} for linear results. The main mechanism of stabilization is the so called ``inviscid damping", which refers to the transfer of energy of vorticity to higher and higher frequencies leading to decay of the stream and velocity functions, as $t\to\infty$.  Extending the linearized stability analysis for inviscid fluid equations to the full nonlinear setting is a challenging problem, and the only available results are on spectrally stable monotonic shear flows \cite{BM,NaZh,IOJI,IJacta}, and on point vortices \cite{IOJI2}. We refer also to the recent review article \cite{IJICM} for a more in-depth discussion of recent developments of both linear and nonlinear inviscid damping. 

Many physically important shear flows are not monotonic, such as the Poiseuille flow and the Kolmogorov flows. For such flows on the linear inviscid level, there is an additional significant physical phenomenon called ``vorticity depletion" which refers to the asymptotic vanishing of vorticity as $t\to\infty$ near the critical point where the derivative of the shear flow is zero, first predicted in Bouchet and Morita \cite{Bouchet}, and proved rigorously in Wei-Zhang-Zhao \cite{Dongyi2}.  A similar phenomenon was proved in Bedrossian-Coti Zelati-Vicol \cite{Bed2} for the case of vortices. See also \cite{JiaVortex} by the first and third author for a refined description of the dynamics in Gevrey spaces as a step towards proving nonlinear vortex symmetrization. 

In \cite{Dongyi2} by Wei-Zhang-Zhao, sharp linear inviscid damping estimates and quantitative depletion estimates were obtained for an important class of ``symmetric shear flows" in a channel (see also \cite{Dongyi3} by Wei-Zhang-Zhao for a similar result for Kolmogorov flow). When no symmetry is assumed, only qualitative bounds are available.  Heuristically the general case should be similar to the symmetric one, since the main vorticity depletion mechanism is completely local and asymptotically all shear flows approach symmetric ones at the (non-degenerate) critical points. However there are significant difficulties in using the approach of \cite{Dongyi2} to extend the quantitative depletion bounds of \cite{Dongyi2} to the general case, as the argument in \cite{Dongyi2} relies heavily on decomposition of functions into odd and even parts, which are specific to symmetric shear flows.

 In this paper we prove linear inviscid damping estimates and quantitative vorticity depletion estimates for a class of stable non-monotonic shear flows with one non-degenerate critical point. The main new features of our results are that we do not need symmetry condition on the background shear flow, and that our formulation on quantitative depletion for vorticity function seems to be new even for general symmetric shear flows (see however Wei-Zhang-Zhao \cite{Dongyi3} which contains a sharp depletion rate at the critical points for Kolmogorov flow), see Theorem \ref{thm} below for the precise statements. We begin with the description of our main equations and theorem.
  
\subsection{Main equations} 
Consider the two dimensional Euler equation linearized around a shear flow $(b(y),0)$, in the periodic channel $ (x,y,t)\in \mathbb{T}\times[0,1]\times[0,\infty)$:
\begin{equation}\label{Main1}
\begin{split}
&\partial_t\omega+b(y)\partial_x\omega-b''(y)u^y=0,\\
&{\rm div}\,u=0\qquad{\rm and}\qquad \omega=-\partial_yu^x+\partial_xu^y,
\end{split}
\end{equation}
with the natural non-penetration boundary condition $u^y|_{y=0,1}=0$. 

For the linearized flow, 
$\int\limits_{\mathbb{T}\times[0,\,1]}u^x(x,y,t)\,dxdy$ and $ \int\limits_{\mathbb{T}\times[0,\,1]}\omega(x,y,t)\,dxdy$
are conserved quantities. In this paper, we will assume that
 $$\int_{\mathbb{T}\times[0,1]}u_0^x(x,y)\,dxdy=\int_{\mathbb{T}\times[0,1]}\omega_0\,dx dy=0.$$
 These assumptions can be dropped by adjusting $b(y)$ with a linear shear flow $C_0y+C_1$.
 Then one can see from the divergence free condition on $u$ that 
there exists a stream function $\psi(t,x,y)$ with $\psi(t,x,0)=\psi(t,x,1)\equiv 0$, such that 
\begin{equation}\label{eqS1}
u^x=-\partial_y\psi,\,\,u^y=\partial_x\psi.
\end{equation}
The stream function $\psi$ can be solved through
\begin{equation}\label{eq:equationStream}
\Delta\psi=\omega, \qquad \psi|_{y=0,1}=0.
\end{equation}
We summarize our equations as follows
\begin{equation}\label{main}
\left\{\begin{array}{ll}
\partial_t\omega+b(y)\partial_x\omega-b''(y)\partial_x\psi=0,&\\
\Delta \psi(t,x,y)=\omega(t,x,y),\qquad \psi(t,x,0)=\psi(t,x,1)=0,&\\
(u^x,u^y)=(-\partial_y\psi,\partial_x\psi),&
\end{array}\right.
\end{equation}
for $ t\ge0, (x,y)\in\mathbb{T}\times[0,1]$. 

Our goal is to understand the long time behavior of $\omega(t)$ as $t\to\infty$, with Sobolev regular initial vorticity $\omega_0$.

\subsection{The main results}
 We describe more precisely the main assumptions and our main conclusion. The main conditions we shall assume on the shear flow $b(y)\in C^{4}([0,1])$ are as follows.

\begin{assumption}\label{Massum}
We assume that the background flow $b(y)\in C^{4}([0,1])$ satisfies the following conditions.
\begin{enumerate}

\item 
\begin{equation}\label{Massum1}S:=\{y\in[0,1]:\,b'(y)=0\}=\{y_\ast\}\subset(0,1).
\end{equation}
In addition, $b''(y_\ast)\neq0$. For the sake of concreteness, we assume that $b''(y_\ast)>0$ in this paper.

\item For $k\in\mathbb{Z}\backslash\{0\}$, the linearized operator $L_k: L^2(0,1)\to L^2(0,1)$ defined as 
\begin{equation}\label{B11}
L_kg(y):=b(y)g(y)+b''(y)\int_0^1G_k(y,z)g(z)\,dz
\end{equation} 
has no discrete eigenvalues nor generalized embedded eigenvalues. In the above $G_k$ is the Green's function for $k^2-\frac{d^2}{dy^2}$ on the interval $(0,1)$ with zero Dirichlet boundary condition, as defined in \eqref{eq:GreenFunction} below. 

\end{enumerate}

\end{assumption}
We refer to section \ref{rep} below for the definition and more discussion about generalized embedded eigenvalues.\\

Our main result is the following theorem.
\begin{theorem}\label{thm}
Assume that $\omega(t,\cdot)\in C([0,\infty), H^4(\T\times[0,1]))$ with the associated stream function $ \psi(t,\cdot)$ is the unique solution to \eqref{main}, with initial data $\omega_0\in H^4(\T\times[0,1])$ satisfying for all $y\in[0,1]$,
\begin{equation}\label{thm0}
\int_\T\omega_0(x,y)\,dx=0.
\end{equation}
Then we have the following bounds.

(i) Inviscid damping estimates:
\begin{equation}\label{th0.5}
\|\psi(t,\cdot)\|_{L^2(\T\times[0,1])}\lesssim \frac{1}{\langle t\rangle^2}\|\omega_0\|_{H^4(\T\times[0,1])},
\end{equation}
\begin{equation}\label{thm1}
\|u^x(t,\cdot)\|_{L^2(\T\times[0,1])}\lesssim \frac{1}{\langle t\rangle}\|\omega_0\|_{H^4(\T\times[0,1])}, \quad \|u^y(t,\cdot)\|_{L^2(\T\times[0,1])}\lesssim \frac{1}{\langle t\rangle^2}\|\omega_0\|_{H^4(\T\times[0,1])}. 
\end{equation} 

(ii) Vorticity depletion estimates: there exists a decomposition
\begin{equation}\label{thm2}
\omega(t,x,y):=\omega_{\rm loc}(t,x,y)+\omega_{\rm nloc}(t,x,y),
\end{equation}
where for $(x,y,t)\in\T\times[0,1]\times[0,\infty)$,
\begin{equation}\label{thm3}
|\omega_{\rm loc}(t,x,y)|\lesssim |y-y_\ast|^{7/4}\|\omega_0\|_{H^4(\T\times[0,1])},\quad |\omega_{\rm nloc}(t,x,y)|\lesssim \frac{1}{\langle t\rangle^{7/8}}\|\omega_0\|_{H^4(\T\times[0,1])}.
\end{equation}

\end{theorem}

\subsection{Remarks and main ideas of proof} We have the following remarks on Theorem \ref{thm}. {\it Firstly}, in the above theorem we have not tracked the minimal regularity required for the bounds \eqref{th0.5}, \eqref{thm1} and \eqref{thm3} to hold, and a more careful argument can probably significantly reduce the number of derivatives needed on the initial data $\omega_0$. {\it Secondly}, we note also that the argument here can be applied to non-monotonic shear flows with multiple non-degenerate points, although the presentation will be more complicated.  {\it Thirdly}, a more sophisticated analysis may yield a sharper rate of vorticity depletion with rate 
$$|\omega_{\rm loc}(t,x,y)|\lesssim |y-y_\ast|^{2-}, \quad  |\omega_{\rm nloc}(t,x,y)|\lesssim \langle t\rangle^{-1+}.$$
It is not clear to us though if one can reach the optimal rates of $|y-y_\ast|^{2}$ and $\langle t\rangle^{-1}$. 

We briefly explain the main ideas of the proof. 

By a standard spectral representation formula, see \eqref{F3.4}, it suffices to study the spectral density functions and the associated Rayleigh equation \eqref{F7}. There are two main cases to consider. When the spectral parameter $\lambda$ is not close to the critical value $b(y_\ast)$, the situation is similar to monotonic shear flows and can be treated as in \cite{JiaL}. The main new case is when the spectral parameter $\lambda$ is close to the critical value $b(y_\ast)$. In this case, the Rayleigh equation \eqref{F7} is very singular, and the potential term $\frac{b''(y)}{b(y)-\lambda+i\epsilon}$ has a quadratic singularity roughly of the form $\frac{2}{(y-y_\ast)^2+(\lambda-b(y_\ast))+i\epsilon}$ for $y$ close to $y_\ast$. 

The key observation here, as in \cite{JiaVortex}, is that the potential term $\frac{b''(y)}{b(y)-\lambda+i\epsilon}$ is {\it critically singular} and has real part with a {\it favorable} sign for $1\gg|y-y_\ast|\gg|\lambda-b(y_\ast)|^{1/2}$, which needs to be incorporated as part of the main term. We therefore define a modified Green's function for the main term, see \eqref{mGk1}-\eqref{mGk2}, which has strong vanishing conditions near $y=y_\ast$, leading ultimately to vorticity depletion. After extracting the main terms in the Rayleigh equation \eqref{F7}, the rest of the terms can be treated as compact perturbations, and can be bounded using a limiting absorption principle, see Lemma \ref{DeC7}, thanks to the spectral assumption \ref{Massum}. 

The limiting absorption principle provides preliminary bounds on the spectral density functions $\psi_{k,\epsilon}^{\iota}(y,\lambda)$ with $\iota\in\{\pm\}$. To obtain the desired quantitative decay rates, we take up to two derivatives in $\lambda$ of the spectral density functions, and again use the limiting absorption principle to estimate the resulting derivatives, after extracting the main singular terms. The procedure is more or less straightforward but the calculations are quite lengthy. We refer to \cite{JiaL} also for similar calculations in a simpler setting. Lastly, we note that there are important cancellations between $\psi_{k,\epsilon}^{+}(y,\lambda)$ and $\psi_{k,\epsilon}^{-}(y,\lambda)$ in the limit $\epsilon\to 0+$, which is the reason why we need two versions of the limiting absorption principle, see Lemma \ref{DeC7}, with different weighted spaces. 

\subsection{Notations} We summarize here some notations that are specific for this paper for the reader's conveniences. For positive numbers $\alpha, \beta$, we set $\alpha\wedge\beta:=\min\{\alpha,\beta\}$. We denote for $d>0$, $\Sigma_d:=\{b(y):\,\,y\in[y_\ast-d,y_\ast+d]\}$, $S_d:=[y_\ast-d,y_\ast+d]$. We also denote $\Sigma:=\{b(y):\,y\in[0,1]\}$ and $I:=[0,1]$. For $k\in\Z\backslash\{0\}$, we define for $f\in H^1(I)$ the norm $\|f\|_{H^1_k(I)}:=\|f\|_{L^2(I)}+|k|^{-1}\|f'\|_{L^2(I)}$. 

%
%

\section{Spectral property and representation formula}\label{rep}
Taking Fourier transform in $x$ in the equation \eqref{main} for $\omega$, we obtain that
 \begin{equation}\label{F3.0}
 \partial_t\omega_k+ikb(y)\omega_k-ikb''(y)\psi_k=0,
 \end{equation}
 for $k\in\mathbb{Z}, t\ge0, y\in[0,1]$. In the above, $\omega_k$ and $\psi_k$ are the $k$-th Fourier coefficients of $\omega,\psi$ in $x$ respectively.
 For each $k\in\mathbb{Z}\backslash\{0\}$, recall from \eqref{B11} that for any $g\in L^2(0,1)$,
 \begin{equation}\label{F3.1}
 L_kg(y)=b(y)g(y)+b''(y)\int_0^1G_k(y,z)g(z)dz,
 \end{equation}
 where $G_k$ is the Green's function for the operator $k^2-\frac{d^2}{dy^2}$ on $(0,1)$ with zero Dirichlet boundary condition. 
Then \eqref{F3.0} can be reformulated abstractly as
 \begin{equation}\label{F3.2}
 \partial_t\omega_k+ikL_k\omega_k=0.
 \end{equation}
 
 In contrast to the spectral property of the linearized operator around monotonic shear flows, the spectral property of $L_k$ is less understood, especially on the generation of discrete eigenvalues and embedded eigenvalues. From general spectral theory, we know that the spectrum of $L_k$ consists of the continuous spectrum 
 \begin{equation}\label{continspe1}
 \Sigma:=\big\{b(y):\,y\in[0,1]\big\},
 \end{equation} 
 together with some discrete eigenvalues with nonzero imaginary part which can only accumulate at the set of continuous spectrum $\Sigma$. Unlike the case of monotonic shear flows where the discrete eigenvalues can accumulate only at inflection points of the background shear flow, there appears no simple characterization of the possible accumulation points for non-monotonic shear flows.
 
 Recall that $\lambda\in\Sigma$ is called an embedded eigenvalue if there exists a nontrivial $g\in L^2(0,1)$, such that 
 \begin{equation}\label{F3.20}
 L_kg=\lambda g. 
 \end{equation}
 
 For non-monotonic shear flows, this definition is too restrictive, as accumulation points of discrete eigenvalues may no longer be embedded eigenvalues. To capture the discrete eigenvalues, we recall the following definition of ``generalized embedded eigenvalues", which can be found already in \cite{Dongyi2}, adapted to our setting. 
 
 \begin{definition}\label{emb1}
 We call $\lambda\in\Sigma$ a generalized embedded eigenvalue, if one of the following conditions is satisfied. 
 \begin{itemize}
 \item  $\lambda$ is an embedded eigenvalue.
 
 \item $\lambda\neq b(y_\ast)$ and there exists a nontrivial $\psi\in H^1_0(0,1): (0,1)\to \mathbb{C}$ such that in the sense of distributions on $(0,1)$,
 \begin{equation}\label{emb2}
 (k^2-\partial_y^2)\psi(y)+{\rm P.V.}\frac{b''(y)\psi(y)}{b(y)-\lambda}+i\pi\sum_{z\in [0,1], \,b(z)=\lambda}\frac{b''(z)\psi(z)}{|b'(z)|}\delta(y-z)=0.
 \end{equation}
 
 \end{itemize}
 
 \end{definition}
 
 We remark that our assumption that the critical point $y_\ast$ of $b(y)$ being non-degenerate implies that the sum in \eqref{emb2} is finite, and that the spectral assumption \ref{Massum} is satisfied if $b''>0$ on $[0,1]$. 
 

 \begin{proposition}\label{F3.3}
 Suppose that $k\in\Z\backslash\{0\}$ and $\omega^k_0\in L^2([0,1])$. Then the stream function $\psi_k(t,y)$ for $k\in\mathbb{Z}\backslash\{0\}, y\in[0,1], t\ge0$ has the representation 
 \begin{equation}\label{F3.4}
 \psi_k(t,y)=-\frac{1}{2\pi i}\lim_{\epsilon\to0+}\int_{\Sigma}e^{-ik\lambda t}\left[\psi_{k,\epsilon}^{-}(y,\lambda)-\psi_{k,\epsilon}^{+}(y,\lambda)\right]d\lambda,
 \end{equation}
 where $\psi_{k,\epsilon}^{\iota}(y,\lambda)$ for $\iota\in\{+,-\}, \,y\in[0,1], \,\lambda\in\Sigma,k\in\mathbb{Z}\backslash\{0\}$,  and  sufficiently small $\epsilon\in[-1/4,1/4]\backslash\{0\}$, are the solutions to
 \begin{equation}\label{F7}
 \begin{split}
  -k^2\psi_{k,\epsilon}^{\iota}(y,\lambda)+\frac{d^2}{dy^2}\psi_{k,\epsilon}^{\iota}(y,\lambda)-\frac{b''(y)}{b(y)-\lambda+i\iota\epsilon}\psi_{k,\epsilon}^{\iota}(y,\lambda)=\frac{-\omega^k_0(y)}{b(y)-\lambda+i\iota\epsilon},
 \end{split}
 \end{equation}
 with zero Dirichlet boundary condition.
 \end{proposition}
 
 \begin{proof}
By standard theory of spectral projection, from \eqref{F3.2}, we obtain that for $y\in[0,1]$,
 \begin{equation}\label{F4}
 \begin{split}
 \omega_k(t,y)&=\frac{1}{2\pi i}\lim_{\epsilon\to0+}\int_{\Sigma}e^{i\lambda t}\left\{\left[(\lambda+kL_k-i\epsilon)^{-1}-(\lambda+kL_k+i\epsilon)^{-1}\right]\omega^k_0\right\}(y)\,d\lambda.
 \end{split}
 \end{equation}
 
We then obtain for $y\in[0,1]$,
 \begin{equation}\label{F5}
 \begin{split}
 \psi_k(t,y)&=-\frac{1}{2\pi i}\lim_{\epsilon\to0+}\int_{\Sigma}e^{-ik\lambda t}\int_0^1G_k(y,z)\\
 &\hspace{1in}\times\bigg\{\Big[(-\lambda+L_k-i\epsilon)^{-1}-(-\lambda+L_k+i\epsilon)^{-1}\Big]\omega^k_0\bigg\}(z)\,dz d\lambda\\
 &=-\frac{1}{2\pi i}\lim_{\epsilon\to0+}\int_{\Sigma}e^{-ik\lambda t}\left[\psi_{k,\epsilon}^{-}(y,\lambda)-\psi_{k,\epsilon}^{+}(y,\lambda)\right]d\lambda.
 \end{split}
 \end{equation}
 In the above, for $y\in[0,1]$ and $\lambda\in\Sigma$,
 \begin{equation}\label{F6}
 \begin{split}
 &\psi_{k,\epsilon}^{+}(y,\lambda):=\int_0^1G_k(y,z)\Big[(-\lambda+L_k+i\epsilon)^{-1}\omega^k_0\Big](z)\,dz,\\
 &\psi_{k,\epsilon}^{-}(y,\lambda):=\int_0^1G_k(y,z)\Big[(-\lambda+L_k-i\epsilon)^{-1}\omega^k_0\Big](z)\,dz.
 \end{split}
 \end{equation}
 Therefore for $\iota\in\{+,-\}, y\in[0,1], \lambda\in\Sigma$,
 \begin{equation}\label{FF6}
 \left(k^2-\frac{d^2}{dy^2}\right)\psi_{k,\epsilon}^{\iota}(y,y_0)=(-\lambda+L_k+i\iota\epsilon)^{-1}\omega^k_0(y),
 \end{equation}
 which implies
 \begin{equation}\label{FF7}
 \begin{split}
 \omega_0^k(y)=&(-\lambda+L_k+i\iota\epsilon)\left(k^2-\frac{d^2}{dy^2}\right)\psi_{k,\epsilon}^{\iota}(y,\lambda)\\
                          =&(b(y)-\lambda+i\iota\epsilon)\left(k^2-\frac{d^2}{dy^2}\right)\psi_{k,\epsilon}^{\iota}(y,\lambda)+b''(y)\psi_{k,\epsilon}^{\iota}(y,\lambda).
 \end{split}
 \end{equation}
It follows from \eqref{FF7} that $\psi_{k,\epsilon}^{+}(y,\lambda), \psi_{k,\epsilon}^{-}(y,\lambda)$ satisfy \eqref{F7}. The proposition is now proved.

\end{proof}

 
\begin{remark}\label{F6.1}
The existence of $\psi^{\iota}_{k,\epsilon}$ for sufficiently small $\epsilon\neq0$ follows from our spectral assumptions, which imply the solvability of \eqref{F7} for sufficiently small $\epsilon\neq0$, see also \eqref{T6}.

\end{remark}

\section{Bounds on the Green's function and modified Green's function}
\subsection{Elementary properties of the standard Green's function}\label{sub1}
For integers $k\in\mathbb{Z}\setminus\{0\}$, recall that the Green's function $G_k(y,z)$ solves  
\begin{equation}\label{eq:Helmoltz}
-\frac{d^2}{dy^2}G_k(y,z)+k^2G_k(y,z)=\delta(y-z),
\end{equation}
with Dirichlet boundary conditions $G_k(0,z)=G_k(1,z)=0$, $z\in (0,1)$. $G_k$ has the explicit formula 
\begin{equation}\label{eq:GreenFunction}
G_k(y,z)=\frac{1}{k\sinh k}
\begin{cases}
\sinh(k(1-z))\sinh (ky)\qquad&\text{ if }y\leq z,\\
\sinh (kz)\sinh(k(1-y))\qquad&\text{ if }y\geq z,
\end{cases}
\end{equation}
and the symmetry
\begin{equation}\label{Gk2}
G_k(y,z)=G_k(z,y), \qquad {\rm for}\,\,k\in\mathbb{Z}\backslash\{0\}, y, z\in[0,1].
\end{equation}

We note the following bounds for $G_k$
\begin{equation}\label{Gk1.1}
\begin{split}
&\sup_{y\in[0,1], |A|\leq10}\bigg[|k|^2\big\|G_k(y,z)(\log{|z-A|})^{m}\big\|_{L^1(z\in[0,1])}+|k|\big\|\partial_{y,z}G_k(y,z)(\log{|z-A|})^{m}\big\|_{L^1(z\in[0,1])}\bigg]\\
&\qquad+\sup_{y\in[0,1],\alpha\in\{0,1\}}\bigg[|k|^{3/2-\alpha}\left\|\partial_{y,z}^{\alpha}G_k(y,z)\right\|_{L^2(z\in[0,1])}\bigg]\lesssim |\log{\langle k\rangle}|^m,\qquad {\rm for}\,\,m\in\{0,1,2,3\}.
\end{split}
\end{equation}

Define
\begin{equation}\label{bX5}
F_k(y,z)=\frac{1}{\sinh{k}}\left\{\begin{array}{lr}
                                         -k\cosh{(k(1-z))}\cosh{(ky)}, &0\leq y\leq z\leq 1;\\
                                         -k\cosh{(kz)}\cosh{(k(1-y))}, &1\ge y>z\ge0.
                                          \end{array}\right.
\end{equation}
We note that
\begin{equation}\label{Gk1.2}
\partial_y\partial_zG_k(y,z)=\partial_z\partial_yG_k(y,z)=\delta(y-z)+F_k(y,z),\qquad{\rm for}\,\,y,z\in[0,1].
\end{equation}

By direct computation, we see $F_k$ satisfies the bounds
\begin{equation}\label{Gk3.1}
\begin{split}
&\sup_{y\in[0,1], |A|\leq10}\bigg[\big\|F_k(y,z)(\log{|z-A|})^{m}\big\|_{L^1(z\in[0,1])}+|k|^{-1}\big\|\partial_{y,z}F_k(y,z)(\log{|z-A|})^{m}\big\|_{L^1(z\in[0,1])}\bigg]\\
&\quad+\sup_{y\in[0,1],\alpha\in\{0,1\}}\bigg[|k|^{-1/2-\alpha}\left\|\partial_{y,z}^{\alpha}F_k(y,z)\right\|_{L^2(z\in[0,1])}\bigg]\lesssim |\log{\langle k\rangle}|^m,\qquad {\rm for}\,\,m\in\{0,1,2,3\}.
\end{split}
\end{equation}

The bounds \eqref{Gk1.1} and \eqref{Gk3.1} can be proved by explicit calculations and are useful in the proof of Lemma \ref{T2} below.


\subsection{Bounds on the modified Green's function}
It follows from Assumption \ref{Massum} that there exists a $\delta_0\in(0,1/8)$ such that
\begin{equation}\label{mGk0}
\inf\{|y_\ast|, |y_\ast-1|\}>10\delta_0 \quad{\rm and}\quad\sup_{y\in(y_\ast-4\delta_0,y_\ast+4\delta_0)}|b'''(y)|\delta_0<|b''(y_\ast)|/10.
\end{equation}

Define the set
\begin{equation}\label{mGk0.2}
\Sigma_{\delta_0}:=\{b(y): y\in[y_\ast-\delta_0,y_\ast+\delta_0]\},
\end{equation}
and fix a standard smooth cutoff function $\varphi\in C_c^\infty(-2,2)$ satisfying $\varphi\equiv 1$ on $[-3/2,3/2]$. For simplicity of notations, we denote
\begin{equation}\label{LAP0.1}
I:=(0,1).
\end{equation}
To simplify notations we define also for $d\in(0,1/10)$,
\begin{equation}\label{LAP0.01}
S_d:=[y_\ast-d,y_\ast+d]. 
\end{equation}

For applications below, we also need to study the ``modified Green's function" $\mathcal{G}_{k}(y,z;\lambda+i\epsilon)$ for $ y,z\in[0,1], \lambda\in \Sigma_{\delta_0}$ and $\epsilon\in[-1/8,1/8]\backslash\{0\}$, which satisfies for $y,z\in(0,1),$
\begin{equation}\label{mGk1}
(k^2-\partial_y^2)\mathcal{G}_{k}(y,z;\lambda+i\epsilon)+\frac{b''(y)}{b(y)-\lambda+i\epsilon}\Big[\varphi\big(\frac{y-y_\ast}{\delta_0}\big)-\varphi\big(\frac{y-y_\ast}{\delta(\lambda)}\big)\Big]\mathcal{G}_{k}(y,z;\lambda+i\epsilon)=\delta(y-z),
\end{equation}
with the boundary condition 
\begin{equation}\label{mGk2}
\mathcal{G}_{k}(y,z;\lambda+i\epsilon)|_{y\in\{0,1\}}=0.
\end{equation}
In the above, we have used the notation that
\begin{equation}\label{mGk3}
\delta(\lambda):=8\sqrt{|\lambda-b(y_\ast)|/b''(y_\ast)}.
\end{equation}

Define the weight 
 $\varrho(y;\lambda+i\epsilon)$ for $y,z\in[0,1], \lambda\in \Sigma_{\delta_0}$ and $\epsilon\in[-1/8,1/8]\backslash\{0\}$ as
\begin{equation}\label{DeC1}
\begin{split}
\varrho(y;\lambda+i\epsilon):=&|\lambda-b(y_\ast)|^{1/2}+|\epsilon|^{1/2}+|y-y_\ast|.
\end{split}
\end{equation}
The crucial bounds we need for the modified Green's function $\mathcal{G}_k(y,z;\lambda+i\epsilon)$ is the following. 
\begin{lemma}\label{mGk4}
Let $\mathcal{G}_{k}(y,z;\lambda+i\epsilon)$ for $y,z\in[0,1], \lambda\in \Sigma_{\delta_0}$ and $\epsilon\in[-1/8,1/8]\backslash\{0\}$ be defined as in \eqref{mGk1}.
Then we have the identity for $y,z\in [0,1]$,
\begin{equation}\label{mGk4.0}
\mathcal{G}_{k}(y,z;\lambda+i\epsilon)=\mathcal{G}_{k}(z,y;\lambda+i\epsilon),
\end{equation}
and the following statements hold. 

(i) We have the bounds
\begin{equation}\label{mGk5}
\begin{split}
&\sup_{y\in[0,1],\, |y-z|\leq \min\{\varrho(z;\lambda+i\epsilon), 1/|k|\}}|\mathcal{G}_{k}(y,z;\lambda+i\epsilon)|\lesssim\min\{\varrho(z;\lambda+i\epsilon), 1/|k|\},\\
&\sup_{y\in[0,1],\, |y-z|\leq \min\{\varrho(z;\lambda+i\epsilon), 1/|k|\}}|\partial_y\mathcal{G}_{k}(y,z;\lambda+i\epsilon)|\lesssim1;
\end{split}
\end{equation}

(ii) For $y_1, y_2\in[0,1]$ with $y_2\in[\min\{y_1,z\}, \max\{y_1,z\}]$ and $\varrho(y_2;\lambda+i\epsilon)\gtrsim 1/|k|$, we have the bounds with $\alpha\in\{0,1\}$
\begin{equation}\label{mGk5.6}
\begin{split}
&|\partial_y^\alpha\mathcal{G}_{k}(y_1,z;\lambda+i\epsilon)|\\
&\lesssim \Big[|k|+\varrho^{-1}(y_1;\lambda+i\epsilon)\Big]^\alpha e^{-|k||y_1-y_2|}\bigg[|k|\int_{[y_2- 1/|k|,y_2+ 1/|k|]\cap I}|\mathcal{G}_{k}(y,z;\lambda+i\epsilon)|^2\,dy\bigg]^{1/2}.
\end{split}
\end{equation}

(iii) For $y_1, y_2\in[0,1]$ with $y_2\in[\min\{y_1,z\}, \max\{y_1,z\}]$ and $\varrho(y_2;\lambda+i\epsilon)\ll 1/|k|$, we have the bounds with $\alpha\in\{0,1\}$
\begin{equation}\label{mGk5.61}
\begin{split}
&|\partial_y^\alpha\mathcal{G}_{k}(y_1,z;\lambda+i\epsilon)|\lesssim \Big[|k|+\varrho^{-1}(y_1;\lambda+i\epsilon)\Big]^\alpha\min\bigg\{\frac{\varrho^2(y_1;\lambda+i\epsilon)}{\varrho^2(y_2;\lambda+i\epsilon)},\,\frac{\varrho(y_2;\lambda+i\epsilon) }{\varrho(y_1;\lambda+i\epsilon)}\bigg\}M,
\end{split}
\end{equation}
where 
\begin{equation}\label{mGk5.7}
 M:=\bigg[\frac{1}{\varrho(y_2;\lambda+i\epsilon)}\int_{[y_2-\varrho(y_2;\lambda+i\epsilon),y_2+\varrho(y_2;\lambda+i\epsilon)]\cap I}|\mathcal{G}_{k}(y,z;\lambda+i\epsilon)|^2\,dy\bigg]^{1/2}.
\end{equation}

\end{lemma}

\begin{proof}

The proof is based on energy estimates and ``entanglement inequalities", as in \cite{JiaUn}. See also the earlier work \cite{Wei3} where this type of inequality was used. We divide the proof into several steps. 

{\bf Step 1: the proof of \eqref{mGk5}.} We first establish the bounds \eqref{mGk5}. 
For simplicity of notation, we suppress the dependence on $z, \lambda+i\epsilon$ and set for $y\in[0,1]$,
\begin{equation}\label{mGk5.8}
h(y):=\mathcal{G}_{k}(y,z;\lambda+i\epsilon), \quad V(y):= \frac{b''(y)}{b(y)-\lambda+i\epsilon}\Big[\varphi\big(\frac{y-y_\ast}{\delta_0}\big)-\varphi\big(\frac{y-y_\ast}{\delta}\big)\Big].
\end{equation}
Multiplying $\overline{h}$ to \eqref{mGk1} and integrating over $[0,1]$, we obtain that 
\begin{equation}\label{mGk6}
\int_0^1|\partial_yh(y)|^2+|k|^2|h(y)|^2\,dy+\int_0^1V(y)|h(y)|^2\,dy=\overline{h}(z). 
\end{equation}
Note that for $y\in [0,1]$, $\Re V(y)\ge0$, and in addition, 
 for  $y\in S_{\delta_0}$ and 
 $$|y-y_\ast|>C_0\big(|\lambda-b(y_\ast)|^{1/2}+|\epsilon|^{1/2}\big)$$ with sufficiently large $C_0\gg1$,
 \begin{equation}\label{mGk8}
1+ \Re V(y)\gtrsim \frac{1}{\varrho^2(y;\lambda+i\epsilon)}.
 \end{equation}
 It follows from \eqref{mGk6} that 
\begin{equation}\label{mGk9}
\begin{split}
&\int_0^1|\partial_yh(y)|^2+|k|^2|h(y)|^2\,dy+\int_{y\in S_{\delta_0},\, \,|y-y_\ast|>C_0(\delta+|\epsilon|^{1/2})} \frac{1}{\big[\varrho(y;\lambda+i\epsilon)\big]^2}|h(y)|^2\,dy\\
&\lesssim |h(z)|. 
\end{split}
\end{equation}
Using the Sobolev type inequality 
\begin{equation}\label{mGk10}
\|h\|_{L^\infty(J)}\lesssim \|h\|_{L^2(J_\ast)}|J|^{-1/2}+\|\partial_yh\|_{L^2(J)}|J|^{1/2},
\end{equation}
for any intervals $J, J_\ast$ with $J_\ast\subseteq J$ and $|J_\ast|\gtrsim |J|$, and choosing the interval $J\subset I$ as an interval containing $z$ with length of the size $C_1\min\{1/|k|, \varrho(z;\lambda+i\epsilon)\}$, we obtain from \eqref{mGk9} that 
\begin{equation}\label{mGk11}
\begin{split}
&\int_0^1|\partial_yh(y)|^2+|k|^2|h(y)|^2\,dy+\int_{y\in S_{\delta_0},\,\,|y-y_{\ast}|>C_0(\delta+|\epsilon|^{1/2})} \frac{1}{\big[\varrho(y;\lambda+i\epsilon)\big]^2}|h(y)|^2\,dy\\
&\lesssim \min\{1/|k|, \varrho(z;\lambda+i\epsilon)\}. 
\end{split}
\end{equation}
The first inequality of \eqref{mGk5} then follows from \eqref{mGk11} and \eqref{mGk10}. To obtain the second inequality in \eqref{mGk5}, we notice that the equation \eqref{mGk1} and the first inequality in \eqref{mGk5} imply pointwise bounds on $\partial_y^2\mathcal{G}_k$ for $y\neq z$. The desired bound in the second inequality of \eqref{mGk5} then follows from interpolation between the pointwise bounds on $\mathcal{G}_k$ and $\partial_y^2\mathcal{G}_k$. 

{\bf Step 2: the proof of \eqref{mGk5.6}.} Denote
\begin{equation}\label{mGk12.1}
M_1:=\bigg[|k|\int_{[y_2- 1/|k|,y_2+ 1/|k|]\cap I}|\mathcal{G}_{k}(y,z;\lambda+i\epsilon)|^2\,dy\bigg]^{1/2}.
\end{equation}
For the sake of concreteness, we assume that $y_1>z$ (so $y_2\in[z,y_1]$). We shall also assume that $y_1-y_2\gg 1/|k|$ as the other case is analogous but easier. For $\varphi\in C^1_p([y_2, 1])$, the space of piecewise $C^1$ functions, with $\varphi(y_2)=0$, we multiply $\varphi^2\overline{h}$ to equation \eqref{mGk1} and integrate over $[y_2,1]$ to obtain that
\begin{equation}\label{mGk13}
\int_{y_2}^1|\partial_yh(y)|^2\varphi^2(y)+2\partial_yh(y)\overline{h(y)}\varphi(y)\partial_y\varphi(y)+|k|^2\varphi^2(y)|h(y)|^2+V(y)|h(y)|^2\varphi^2(y)\,dy=0.
\end{equation}
Taking the real part of \eqref{mGk13} and using Cauchy-Schwarz inequality, we get that 
\begin{equation}\label{mGk14}
\int_{y_2}^1\big[|\partial_y\varphi(y)|^2-|k|^2|\varphi(y)|^2\big]|h(y)|^2\,dy\ge0.
\end{equation}
We now choose $\varphi$ more specifically as follows. We require that 
\begin{equation}\label{mGk15}
\begin{split}
&\varphi(y_2)=0, \,\, \varphi''(y)=0\,\,{\rm for}\,\,y\in[y_2,y_2+1/|k|], \,\,\varphi(y_2+1/|k|)=1,\\ 
&\varphi'(y)=|k|\varphi(y)\,\,{\rm for}\,\,y\in[y_2+1/|k|, y_1-1/|k|],\,\,\varphi'(y)=0\,\,{\rm for}\,\,y\in[y_1-1/|k|, 1]. 
\end{split}
\end{equation}
It follows from \eqref{mGk14}-\eqref{mGk15} that 
\begin{equation}\label{mGk16}
\begin{split}
&\int_{y_1-1/|k|}^{1}|k|^2\varphi^2(y)|h(y)|^2\,dy\lesssim |k|M_1^2,\quad\varphi(y)\approx e^{|k||y_1-y_2|} \,\,{\rm for}\,\,y\in[y_1-1/|k|,y_1+1/|k|]\cap I.
\end{split}
\end{equation}
The desired bounds \eqref{mGk5.6} follow from \eqref{mGk16} and equation \eqref{mGk1}. 

{\bf Step 3: the the proof of \eqref{mGk5.61}.} For the sake of concreteness, we assume that $y_1>z$ (and so $y_2\in[z,y_1]$). We shall also assume that $y_1-y_2\gg \varrho(y_2;\lambda+i\epsilon)$ and that $y_2>y_\ast+\delta+|\epsilon|^{1/2}$ as the other cases are analogous. 

For $\varphi\in C^1_p([y_2, 1])$ with $\varphi(y_2)=0$, we multiply $\varphi^2\overline{h}$ to equation \eqref{mGk1} and integrate over $[y_2,1]$ to obtain that
\begin{equation}\label{mGk17}
\int_{y_2}^1|\partial_yh(y)|^2\varphi^2(y)+2\partial_yh(y)\overline{h(y)}\varphi(y)\partial_y\varphi(y)+|k|^2\varphi^2(y)|h(y)|^2+V(y)|h(y)|^2\varphi^2(y)\,dy=0.
\end{equation}
Write for $y\in[y_2,1]$
\begin{equation}\label{mGk18}
h(y)=(y-y_{\ast})^{1/2}h^\ast(y).
\end{equation}
Simple calculations show that 
\begin{equation}\label{mGk19}
\begin{split}
&\int_{y_2}^1(y-y_{\ast})|\partial_yh^\ast(y)|^2\varphi^2(y)+2(y-y_{\ast})\partial_y\varphi(y)\varphi(y)\partial_yh^\ast(y)\overline{h^\ast(y)}+\frac{1}{4(y-y_{\ast})}|h^\ast(y)|^2\varphi^2(y)\\
&\qquad+|k|^2|h(y)|^2\varphi^2(y)+(y-y_{\ast})V(y)\varphi^2(y)|h^\ast(y)|^2\,dy=0.
\end{split}
\end{equation}
Therefore
\begin{equation}\label{mGk20}
\int_{y_2}^1\Big[\frac{1}{4(y-y_{\ast})}+(y-y_{\ast})\Re V(y)\Big]\varphi^2(y)|h^\ast(y)|^2\,dy\leq \int_{y_2}^1(y-y_{\ast})(\partial_y\varphi)^2(y)|h^\ast(y)|^2\,dy,
\end{equation}
which implies that
\begin{equation}\label{mGk21}
\int_{y_2}^1\frac{1}{y-y_{\ast}}\Big[\big((y-y_{\ast})\partial_y\varphi\big)^2(y)-\Big(1/4+(y-y_{\ast})^2\Re V(y)\Big)\varphi^2(y)\Big]|h^\ast(y)|^2\,dy\ge0.
\end{equation}
We notice the pointwise bounds for $y\in[y_2, 1]$, 
\begin{equation}\label{mGk22}
1/4+(y-y_{\ast})^2\Re V(y)\ge \max\Big\{0, 9/4-C_2\frac{\varrho^2(y_2;\lambda+i\epsilon)}{(y-y_{\ast})^2}-C_2|y-y_{\ast}|\Big\}.
\end{equation}
Now we choose $\varphi\in C^1_p([y_2,1])$ more precisely as follows. We require that 
\begin{equation}\label{mGk23}
\begin{split}
&\varphi(y_2)=0, \,\, \varphi''(y)=0\,\,{\rm for}\,\,y\in[y_2,y_2+\varrho(y_2;\lambda+i\epsilon)], \,\,\varphi(y_2+\varrho(y_2;\lambda+i\epsilon))=1,\\ 
&(y-y_{\ast})\varphi'(y)=\big[1/4+(y-y_{\ast})^2\Re V(y)\big]^{1/2}\varphi(y)\\
&{\rm for}\,\,y\in[y_2+\varrho(y_2;\lambda+i\epsilon), y_1-\varrho(y_1;\lambda+i\epsilon)],\,\,{\rm and}\,\,\varphi'(y)=0\,\,{\rm for}\,\,y\in[y_1-\varrho(y_1;\lambda+i\epsilon), 1]. 
\end{split}
\end{equation}
It follows from \eqref{mGk21}-\eqref{mGk23} that 
\begin{equation}\label{mGk24}
\begin{split}
&\int_{y_1-\varrho(y_1;\lambda+i\epsilon)}^{y_1}\frac{1}{\varrho(y_1;\lambda+i\epsilon)}\varphi^2(y)|h^\ast(y)|^2\,dy\lesssim M^2/\varrho(y_2;\lambda+i\epsilon),\\
&\varphi(y)\approx \frac{(y_1-y_{\ast})^{3/2}}{\varrho^{3/2}(y_2;\lambda+i\epsilon)} \,\,{\rm for}\,\,y\in[y_1-\varrho(y_1;\lambda+i\epsilon),y_1].
\end{split}
\end{equation}
The desired bounds \eqref{mGk5.61} follow from the change of variable \eqref{mGk18}, the bound \eqref{mGk21}, \eqref{mGk24} and equation \eqref{mGk1}. 

Lastly we indicate how to prove the identity \eqref{mGk4.0}. For any $y,z\in (0,1)$, using the notation in \eqref{mGk5.8} for $V$, we have by integration by parts
\begin{equation}\label{mGk5.771}
\begin{split}
\mathcal{G}_k(y,z;\lambda+i\epsilon)&=\int_{[0,1]}(k^2-\partial_\rho^2+V(\rho))\mathcal{G}_k(\rho,y;\lambda+i\epsilon)\,\mathcal{G}_k(\rho,z;\lambda+i\epsilon)\,d\rho\\
&=\int_{[0,1]}(k^2-\partial_\rho^2+V(\rho))\mathcal{G}_k(\rho,z;\lambda+i\epsilon)\,\mathcal{G}_k(\rho,y;\lambda+i\epsilon)\,d\rho\\
&=\mathcal{G}_k(z,y;\lambda+i\epsilon),
\end{split}
\end{equation}
which completes the proof of \eqref{mGk4.0}.
\end{proof}

As a corollary of Lemma \ref{mGk4}, we have the following additional bounds on the modified Green's function.

\begin{lemma}\label{mGk30}
Let $\mathcal{G}_{k}(y,z;\lambda+i\epsilon)$ for $y,z\in[0,1], \lambda\in \Sigma_{\delta_0}, k\in\Z\backslash\{0\}$ and $\epsilon\in[-1/8,1/8]\backslash\{0\}$ be defined as in \eqref{mGk1}. Recall the definition \eqref{mGk3} for $\delta=\delta(\lambda)>0$. 
 Define
 \begin{equation}\label{mGk30.1}
 h:=10(\delta+|\epsilon|^{1/2}),
 \end{equation}
  and also for $y,z\in [0,1]$,
 \begin{equation}\label{mGk31}
\mathcal{H}_{k}(y,z;\lambda+i\epsilon):=\Big[\partial_z+\varphi\big(\frac{y-y_{\ast}}{h}\big)\partial_y\Big]\mathcal{G}_{k}(y,z;\lambda+i\epsilon).
\end{equation}
Then the following statements hold for $z\in S_{4\delta}$.

(i) We have the bounds
\begin{equation}\label{mGk32}
\begin{split}
&\sup_{y\in[0,1],\, |y-z|\leq \min\{\varrho(z;\lambda+i\epsilon), 1/|k|\}}|\mathcal{H}_{k}(y,z;\lambda+i\epsilon)|\lesssim1,\\
&\sup_{y\in[0,1],\, |y-z|\leq \min\{\varrho(z;\lambda+i\epsilon), 1/|k|\}}|\partial_y\mathcal{H}_{k}(y,z;\lambda+i\epsilon)|\lesssim1/\min\{\varrho(z;\lambda+i\epsilon), 1/|k|\};
\end{split}
\end{equation}

(ii) For $y_1, y_2\in[0,1]$ with $y_2\in[\min\{y_1,z\}, \max\{y_1,z\}]$ and $\varrho(y_2;\lambda+i\epsilon)\gtrsim 1/|k|$, we have the bounds with $\alpha\in\{0,1\}$
\begin{equation}\label{mGk33}
\begin{split}
&\big[\min\{\varrho(y_1;\lambda+i\epsilon), 1/|k|\}\big]^\alpha|\partial^\alpha_y\mathcal{H}_{k}(y_1,z;\lambda+i\epsilon)|\\
&\lesssim  \frac{e^{-|k||y_1-y_2|}}{\min\{\varrho(z;\lambda+i\epsilon), 1/|k|\}}\bigg[|k|\int_{[y_2- 1/|k|,y_2+ 1/|k|]\cap I}|\mathcal{G}_{k}(y,z;\lambda+i\epsilon)|^2\,dy\bigg]^{1/2}.
\end{split}
\end{equation}

(iii) For $y_1, y_2\in[0,1]$ with $y_2\in[\min\{y_1,z\}, \max\{y_1,z\}]$ and $\varrho(y_2;\lambda+i\epsilon)\ll 1/|k|$, we have the bounds with $\alpha\in\{0,1\}$
\begin{equation}\label{mGk35}
\begin{split}
&\big[\min\{\varrho(y_1;\lambda+i\epsilon), 1/|k|\}\big]^\alpha|\partial^\alpha_y\mathcal{H}_{k}(y_1,z;\lambda+i\epsilon)|\\
&\lesssim \frac{1}{\min\{\varrho(z;\lambda+i\epsilon), 1/|k|\}} \min\bigg\{\frac{\varrho^2(y_1;\lambda+i\epsilon)}{\varrho^2(y_2;\lambda+i\epsilon)},\,\frac{\varrho(y_2;\lambda+i\epsilon)}{\varrho(y_1;\lambda+i\epsilon)}\bigg\} M,
\end{split}
\end{equation} 
where 
\begin{equation}\label{mGk37}
 M:=\bigg[\frac{1}{\varrho(y_2;\lambda+i\epsilon)}\int_{[y_2-\varrho(y_2;\lambda+i\epsilon),y_2+\varrho(y_2;\lambda+i\epsilon)]\cap I}|\mathcal{G}_{k}(y,z;\lambda+i\epsilon)|^2\,dy\bigg]^{1/2}.
\end{equation}

\end{lemma}

\begin{proof}
Denote with a slight abuse of notation for $y\in [0,1]$, 
\begin{equation}\label{mGk38}
\varphi^\dagger(y):=\varphi\big(\frac{y-y_{\ast}}{h}\big),\quad V(y):=\frac{b''(y)}{b(y)-\lambda+i\epsilon}\Big[\varphi\big(\frac{y-y_\ast}{\delta_0}\big)-\varphi\big(\frac{y-y_\ast}{\delta(\lambda)}\big)\Big].
\end{equation}
Then $\mathcal{H}_{k,j}(y,z;\lambda+i\epsilon)$ satisfies for $y\in[0,1], z\in S_{4\delta}$,
\begin{equation}\label{mGk39}
\begin{split}
&(k^2-\partial_y^2)\mathcal{H}_{k}(y,z;\lambda+i\epsilon)+V(y)\mathcal{H}_{k}(y,z;\lambda+i\epsilon)\\
&=-\partial_y^2\varphi^\dagger(y)\partial_y\mathcal{G}_{k}(y,z;\lambda+i\epsilon)-\partial_yV(y)\varphi^\dagger(y)\mathcal{G}_{k}(y,z;\lambda+i\epsilon)-2\partial_y\varphi^\dagger(y)\partial_y^2\mathcal{G}_{k}(y,z;\lambda+i\epsilon).
\end{split}
\end{equation}
The desired bounds then follow from equation \eqref{mGk39}, Lemma \ref{mGk4} and standard elliptic regularity theory. 

\end{proof}

The bounds in Lemma \ref{mGk4} and Lemma \ref{mGk30} are quite sharp, since we can exploit the decay coming from both $k^2$ and $\frac{b''(y)}{b(y)-\lambda+i\epsilon}\Big[\varphi\big(\frac{y-y_\ast}{\delta_0}\big)-\varphi\big(\frac{y-y_\ast}{\delta(\lambda)}\big)\Big]$. It is however somewhat complicated to formulate a concrete bound that is easy to use. Instead, the following simple bounds are more often used. 
\begin{corollary}\label{mGk50}
Let $\mathcal{G}_{k}(y,z;\lambda+i\epsilon)$ for $y,z\in[0,1], \lambda\in \Sigma_{\delta_0}$ and $\epsilon\in[-1/8,1/8]\backslash\{0\}$ be defined as in \eqref{mGk1}.
Then we have the following bounds. 

(i) 
For $y, z\in[0,1]$, we have the bounds with $\alpha\in\{0,1\}$
\begin{equation}\label{mGk51}
\begin{split}
&\Big[|k|+\varrho^{-1}(y;\lambda+i\epsilon)\Big]^{-\alpha}|\partial_y^\alpha\mathcal{G}_{k}(y,z;\lambda+i\epsilon)|\\
&\lesssim\frac{1}{|k|+\varrho^{-1}(z;\lambda+i\epsilon)} \min\bigg\{e^{-|k||y-z|}, \frac{\varrho^2(y;\lambda+i\epsilon)}{\varrho^2(z;\lambda+i\epsilon)},\,\frac{\varrho(z;\lambda+i\epsilon) }{\varrho(y;\lambda+i\epsilon)}\bigg\}.
\end{split}
\end{equation}

(iii) For $y\in[0,1], z\in S_{4\delta}$, we have the bounds with $\alpha\in\{0,1,2\}$
\begin{equation}\label{mGk52}
\begin{split}
&\Big[|k|+\varrho^{-1}(y;\lambda+i\epsilon)\Big]^{-\alpha}|\partial_y^\alpha\mathcal{H}_{k}(y,z;\lambda+i\epsilon)|\lesssim \min\bigg\{e^{-|k||y-z|}, \frac{\varrho^2(y;\lambda+i\epsilon)}{\varrho^2(z;\lambda+i\epsilon)},\,\frac{\varrho(z;\lambda+i\epsilon) }{\varrho(y;\lambda+i\epsilon)}\bigg\}.
\end{split}
\end{equation}

\end{corollary}

\begin{proof}
The desired bounds \eqref{mGk51}-\eqref{mGk52} follow directly from Lemma \ref{mGk4} and Lemma \ref{mGk30}, by choosing, if necessary, another point $y'$ between $y$ and $z$ such that $\varrho(y';\lambda+i\epsilon)\approx 1/|k|$, and applying \eqref{mGk51}-\eqref{mGk52} on intervals $[\min\{z,y'\},\max\{z,y'\}]$ and $[\min\{y',y\},\max\{y',y\}]$ successively.
\end{proof}

\section{The limiting absorption principle}
In this section we study the solvability of the main Rayleigh equations \eqref{F7}. It turns out that the situation is very different for the spectral range $\lambda\in \Sigma\backslash\Sigma_{\delta_0/2}$ (the non-degenerate case) and $\lambda\in\Sigma_{\delta_0}$ (the degenerate case). We first consider the non-degenerate case. 
\subsection{The non-degenerate case}
Fix $\epsilon\in[-1/4,1/4]\backslash\{0\}, \lambda\in \Sigma\backslash\Sigma_{\delta_0/2},  k\in\mathbb{Z}\backslash\{0\}$. Define for each $g\in L^2(0,1)$ the operator
\begin{equation}\label{T1}
T_{k,\lambda,\epsilon}g(y):=\int_{0}^1G_k(y,z)\frac{b''(z)g(z)}{b(z)-\lambda+i\epsilon}dz.
\end{equation}

For applications below, we fix a smooth cutoff function $\Phi\in C_0^\infty(y_\ast-\delta_0/3,y_\ast+\delta_0/3)$ with $\Phi\equiv 1$ on $[y_\ast-\delta_0/4,y_\ast+\delta_0/4]$.
To obtain the optimal dependence on the frequency variable $k$, we define
\begin{equation}\label{T1.1}
\|g\|_{H^1_k(I)}:=\|g\|_{L^2(I)}+|k|^{-1}\|g'\|_{L^2(I)}.
\end{equation}

\begin{lemma}\label{T2}
For  $\epsilon\in[-1/4,1/4]\backslash\{0\}, \lambda\in \Sigma\backslash\Sigma_{\delta_0/2},  k\in\mathbb{Z}\backslash\{0\}$, the operator $T_{k,\lambda,\epsilon}$ satisfies the bound
\begin{equation}\label{T3}
\|T_{k,\lambda,\epsilon}g\|_{H^1_k(I)}\lesssim |k|^{-1/3}\|g\|_{H^1_k(I)},\qquad {\rm for\,\,all}\,\,g\in H^{1}_k(I).
\end{equation}
In addition, we have the more precise regularity structure
\begin{equation}\label{T3.1}
\begin{split}
&\bigg\|\partial_yT_{k,\lambda,\epsilon}g(y)+\frac{b''(y)(1-\Phi(y))g(y)}{b'(y)}\log{(b(y)-\lambda+i\epsilon)}\bigg\|_{W^{1,1}(\R)}\\
&\lesssim \langle k\rangle^{4/3}\|g\|_{H^1_k(I)}.
\end{split}
\end{equation}

\end{lemma}

\begin{proof}
We can decompose for $y\in [0,1]$,
\begin{equation}\label{T3.2}
T_{k,\lambda,\epsilon}g(y):=T^1_{k,\lambda,\epsilon}g(y)+T^2_{k,\lambda,\epsilon}g(y),
\end{equation}
where
\begin{equation}\label{T3.3}
T^1_{k,\lambda,\epsilon}g(y):=\int_{0}^1G_k(y,z)\frac{\Phi(z)b''(z)g(z)}{b(z)-\lambda-i\epsilon}dz,\quad T^2_{k,\lambda,\epsilon}g(y):=\int_{0}^1G_k(y,z)\frac{(1-\Phi(z))b''(z)g(z)}{b(z)-\lambda+i\epsilon}dz.
\end{equation}
It follows from the definition of $\Phi$ that $T^1_{k,\lambda,\epsilon}g(y)$ satisfies the bound
\begin{equation}\label{T3.4}
\|T^1_{k,\lambda,\epsilon}g(y)\|_{H^1_k(I)}\lesssim |k|^{-1/3}\|g\|_{H^1_k(I)},\quad \|\partial_yT^1_{k,\lambda,\epsilon}g(y)\|_{W^{1,1}(I)}\lesssim  \langle k\rangle^{4/3}\|g\|_{H^1_k(I)}.
\end{equation}
To bound $T^2_{k,\lambda,\epsilon}g(y)$, we follow the approach in \cite{JiaL}. Using integration by parts, we obtain that
\begin{equation}\label{T3.5}
\begin{split}
T^2_{k,\lambda,\epsilon}g(y)&=\int_0^1G_k(y,z)\frac{(1-\Phi(z))b''(z)g(z)}{b'(z)}\partial_z\log(b(z)-\lambda+i\epsilon)\,dz\\
&=-\int_0^1\partial_zG_k(y,z)\frac{(1-\Phi(z))b''(z)g(z)}{b'(z)}\log(b(z)-\lambda+i\epsilon)\,dz\\
&\quad-\int_0^1G_k(y,z)\partial_z\Big[\frac{(1-\Phi(z))b''(z)g(z)}{b'(z)}\Big]\log(b(z)-\lambda+i\epsilon)\,dz.
\end{split}
\end{equation}
The desired bounds follow from the bound \eqref{Gk1.1}, the formula \eqref{Gk1.2} and \eqref{Gk3.1}. 
\end{proof}

We now prove the limiting absorption principle, using the assumption that there is no discrete or generalized embedded eigenvalues.

\begin{lemma}\label{T5}
There exist $\epsilon_0, \kappa>0$, such that the following statement holds. For all $\lambda\in \Sigma\backslash\Sigma_{\delta_0/2}, k\in\mathbb{Z}\backslash\{0\}, 0<|\epsilon|<\epsilon_0$ and any $g\in H^1_k(I)$, we have the bound
\begin{equation}\label{T6}
\|g+T_{k,\lambda,\epsilon}g\|_{H^1_k(I)}\ge\kappa \|g\|_{H^1_k(I)}.
\end{equation}
\end{lemma}

\begin{proof}
We prove \eqref{T6} by contradiction. Assume that there exist for $j\ge1$, a sequence of numbers $k_j\in\mathbb{Z}\backslash\{0\}$, $\lambda_j\in \Sigma\backslash\Sigma_{\delta_0/2}$, $\epsilon_j\in\mathbb{\R}\backslash\{0\}\to 0$ and functions  $g_j\in H_{k_j}^1(I)$ with $\|g_j\|_{H_{k_j}^1(I)}=1$, satisfying 
$k_j\to k_{\ast}\in(\mathbb{Z}\backslash\{0\})\cup \{\pm\infty\}$, $\lambda_j\to \lambda_{\ast}\in \overline{\Sigma\backslash\Sigma_{\delta_0}}$ as $j\to\infty$, such that
\begin{equation}\label{T6.1}
\left\|g_j+T_{k_j,\lambda_j,\epsilon_j}g_j\right\|_{H_{k_j}^1(I)}\to 0,\qquad{\rm as}\,\,j\to\infty.
\end{equation}
The bounds \eqref{T3} and \eqref{T6.1} imply that $|k_j|\lesssim 1$. Thus $k_{\ast}\in \mathbb{Z}\backslash\{0\}$. Using $\|g_j\|_{H_{k_j}^1(I)}=1$, the bounds \eqref{T3.1} and the compact embedding $W^{1,1}(I)\to L^2(I)$, we conclude that by passing to a subsequence, $T_{k_j,\lambda_j,\epsilon_j}g_j$ converges in $H^1(I)$. In view of \eqref{T6.1} we can assume that $g_j\to g$ in $H^1(I)$, where $\|g\|_{H^1_{k_{\ast}}}=1$.

Using formula \eqref{T1}, we obtain from \eqref{T6.1} that for $y\in I$,
\begin{equation}\label{T6.2}
g(y)+\lim_{j\to\infty}\int_{0}^1G_{k_{\ast}}(y,z)\frac{b''(z)g(z)}{b(z)-\lambda+i\epsilon_j}\,dz=0.
\end{equation}
Applying $k_{\ast}^2-\frac{d^2}{dy^2}$ to \eqref{T6.2}, we get that for $y\in I$,
\begin{equation}\label{F6.4}
k_{\ast}^2g(y)-g''(y)+\lim_{j\to\infty}\frac{(b(y)-\lambda_\ast)b''(y)g(y)}{(b(y)-\lambda_\ast)^2+\epsilon_j^2}+i\pi \sum_{z\in[0,1], b(z)=\lambda}\frac{b''(z)g(z)}{|b'(z)|}\delta(y-z)=0,
\end{equation}
in the sense of distributions for $y\in(0,1)$, which contradicts our spectral assumption that $\lambda_\ast$ is not a generalized embedded eigenvalue for $L_k$. The lemma is then proved.
\end{proof}


\subsection{The degenerate case $\lambda\in \Sigma_{\delta_0}$ }

Recall the definition \eqref{mGk3} for $\delta=\delta(\lambda)$. For $\lambda\in \Sigma_{\delta_0}, k\in\Z\backslash\{0\}, y\in I$ and $\epsilon\in[-1/8,1/8]\backslash\{0\}$, we denote
\begin{equation}\label{DeC1.92}
d_k(\lambda,\epsilon):=\big[|\lambda-b(y_\ast)|^{1/2}+|\epsilon|^{1/2}\big]\wedge \frac{1}{|k|},\quad \varrho_k(y;\lambda+i\epsilon):=\varrho(y;\lambda+i\epsilon)\wedge \frac{1}{|k|}.
\end{equation}  
Define the weight and the associated weighted Sobolev spaces $X_{N,\varrho_k}$ and $X_{L,\varrho_k}$ as
\begin{equation}\label{DeC2}
\begin{split}
\|g\|_{X_{N,\varrho_k}(I)}:=&\sum_{\alpha\in\{0,1\}}(\delta+|\epsilon|^{1/2})^{-1/2}\Big\|\big[d_k(\lambda,\epsilon)\big]^{(-7/4+\alpha)}\partial_y^\alpha g\Big\|_{L^2(S_{3(\delta+|\epsilon|^{1/2})})}\\
&+\sum_{\alpha\in\{0,1\}}\|\varrho_k^{-7/4+\alpha}(\cdot;\lambda+i\epsilon)\partial^\alpha_yg\|_{L^\infty(I\backslash S_{3(\delta+|\epsilon|^{1/2})})},
\end{split}
\end{equation}
and
\begin{equation}\label{DeC2.0}
\begin{split}
\|g\|_{X_{L,\varrho_k}(I)}:=&\sum_{\alpha\in\{0,1\}}(\delta+|\epsilon|^{1/2})^{-1/2}\big\|d_k^{\alpha}(\lambda,\epsilon)\partial_y^\alpha g\big\|_{L^2(S_{3(\delta+|\epsilon|^{1/2})})}\\
&+\sum_{\alpha\in\{0,1\}}\big\|d_k(\lambda,\epsilon)^{-1}\varrho_k^{\alpha+1}(\cdot;\lambda+i\epsilon)\partial_y^\alpha g\big\|_{L^\infty(I\backslash S_{3(\delta+|\epsilon|^{1/2})})},
\end{split}
\end{equation}

We remark that the choice of the weights in \eqref{DeC2}-\eqref{DeC2.0} is closely related to the behavior of the modified Green's function $\mathcal{G}_k(y,z;\lambda+i\epsilon)$. In \eqref{DeC2}, we consider the case that the ``source" $z$ is of unit distance away from the critical point $y_\ast$, where the expected decay of $\mathcal{G}_k(y,z;\lambda+i\epsilon)$ towards the $y=y_\ast$ is given roughly by $\varrho_k^2(y;\lambda+i\epsilon)$; similar considerations apply in \eqref{DeC2.0} if one considers the case that the source $z$ is near $y_\ast$ and study the behavior of $\mathcal{G}_k(y,z;\lambda+i\epsilon)$ away from $y_\ast$. The choice of exponent as $7/4$ is somewhat arbitrary, as long as it is less than $2$. The endpoint case of exponent being $2$ though results in a subtle logarithmic divergence that seems more technical to handle. 

Fix $\epsilon\in[-1/4,1/4]\backslash\{0\}, \lambda\in \Sigma_{\delta_0},  k\in\mathbb{Z}\backslash\{0\}$. Recall the definition \eqref{mGk3} for $\delta=\delta(\lambda)>0$. Define for each $g\in L^2(0,1)$ the operator
\begin{equation}\label{DeC3}
T^\ast_{k}(\lambda+i\epsilon)g(y):=\int_{0}^1\mathcal{G}_{k}(y,z;\lambda+i\epsilon)\bigg[1-\varphi\big(\frac{y-y_{\ast}}{\delta_0}\big)+\varphi\big(\frac{y-y_{\ast}}{\delta}\big)\bigg]\frac{b''(z)g(z)}{b(z)-\lambda+i\epsilon}dz.
\end{equation}

Then we have the following bounds for $T^\ast_{k}(\lambda+i\epsilon)$.

\begin{lemma}\label{DeC4}
 For  $\epsilon\in[-1/4,1/4]\backslash\{0\}, \lambda\in \Sigma_{\delta_0},  k\in\mathbb{Z}\backslash\{0\},$ the operator $T^\ast_{k}(\lambda+i\epsilon)$ satisfies the bound for $X\in\{X_{N,\varrho_k}(I),X_{L,\varrho_k}(I)\}$
\begin{equation}\label{DeC5}
\|T^\ast_{k}(\lambda+i\epsilon)g\|_{X}\lesssim (1+|k|(|\lambda-b(y_\ast)|^{1/2}+|\epsilon|^{1/2}))^{-1/4}\|g\|_{X},\quad {\rm for\,\,all}\,\,g\in H^{1}_k(I).
\end{equation}
\end{lemma}

\begin{proof}
We provide the detailed proof only for the case $X=X_{N,\varrho_k}(I)$ as the other case is analogous. Since $k,\lambda, \epsilon$ are fixed, for simplicity of notations, we suppress the dependence on $k,\lambda, \epsilon$ to write $T^\ast$ as $T^\ast_{k}(\lambda+i\epsilon)$, and decompose for $y\in I$, 
\begin{equation}\label{DeC5.1}
T^\ast g(y):=T_1^\ast g(y)+T^\ast_2g(y),
\end{equation}
where 
\begin{equation}\label{DeC5.2}
\begin{split}
&T^\ast_1g(y):=\int_{0}^1\mathcal{G}_{k}(y,z;\lambda+i\epsilon)\bigg[1-\varphi\big(\frac{z-y_{\ast}}{\delta_0}\big)\bigg]\frac{b''(z)g(z)}{b(z)-\lambda+i\epsilon}dz,\\
&T^\ast_2g(y):=\int_{0}^1\mathcal{G}_{k}(y,z;\lambda+i\epsilon)\varphi\big(\frac{z-y_{\ast}}{\delta}\big)\frac{b''(z)g(z)}{b(z)-\lambda+i\epsilon}dz.
\end{split}
\end{equation}
It follows from the bounds on modified Green's function $\mathcal{G}_{k}(y,z;\lambda+i\epsilon)$, see Lemma \ref{mGk4}, that 
\begin{equation}\label{DeC5.3}
\big\|T^\ast_1g\big\|_{X_{N,\varrho_k}(I)}\lesssim |k|^{-1/2}\big\|g\big\|_{X_{N,\varrho_k}(I)}.
\end{equation}
To prove \eqref{DeC5}, it suffices to prove 
\begin{equation}\label{DeC5.4}
\|T^\ast_2g\|_{X_{N,\varrho_k}(I)}\lesssim \big(1+|k|(\delta+|\epsilon|^{1/2})\big)^{-1/4}\|g\|_{X_{N,\varrho_k}(I)}.
\end{equation}

We assume momentarily that $|\epsilon|\lesssim |\lambda-b(y_\ast)|$ and explain how to remove this assumption at the end of the proof.  
We decompose further for $y\in I$,
\begin{equation}\label{DeC5.5}
\begin{split}
T_2^\ast g(y)&=\int_{0}^1\mathcal{G}_{k}(y,z;\lambda+i\epsilon)\varphi\big(\frac{z-y_{\ast}}{\delta'}\big)\varphi\big(\frac{z-y_{\ast}}{\delta}\big)\frac{b''(z)g(z)}{b(z)-\lambda+i\epsilon}dz\\
&+\int_{0}^1\mathcal{G}_{k}(y,z;\lambda+i\epsilon)\Big[1-\varphi\big(\frac{z-y_{\ast}}{\delta'}\big)\Big]\varphi\big(\frac{z-y_{\ast}}{\delta}\big)\frac{b''(z)g(z)}{b(z)-\lambda+i\epsilon}dz\\
&:=T^\ast_{2,R}g(y)+T^\ast_{2,S}g(y),
\end{split}
\end{equation}
where we have chosen $\delta'=\delta/C_3$ with a large constant $C_3$ so that $|b(y)-\lambda|\approx |\lambda-b(y_{\ast})|$ for $|y-y_{\ast}|<\delta'$. 

It suffices to prove for $\diamond\in\{R,S\}$
\begin{equation}\label{DeC5.50}
\|T^\ast_{2,\diamond}g\|_{X_{N,\varrho_k}(I)}\lesssim \big(1+|k|(|\lambda-b(y_\ast)|^{1/2}+|\epsilon|^{1/2})\big)^{-1/4}\|g\|_{X_{N,\varrho_k}(I)}. 
\end{equation}

{\bf Step 1.} We first prove \eqref{DeC5.50} with $\diamond=R$. 

{\it Case I: $1/|k|>|\lambda-b(y_{\ast})|^{1/2}+|\epsilon|^{1/2}$}.  In this case for $|z-y_{\ast}|\lesssim\delta$ and $|y-y_{\ast}|\lesssim1$ we have the bound 
\begin{equation}\label{DeC5.6}
\big|\mathcal{G}_{k}(y,z;\lambda+i\epsilon)\big|\lesssim\frac{\delta^2+|\epsilon|}{|y-y_{\ast}|+\delta+|\epsilon|^{1/2}},\quad \big|\partial_y\mathcal{G}_{k}(y,z;\lambda+i\epsilon)\big|\lesssim\frac{\delta^2+|\epsilon|}{(|y-y_{\ast}|+\delta+|\epsilon|^{1/2})^2}. 
\end{equation}
It follows from the bound \eqref{DeC5.6} that 
\begin{equation}\label{DeC5.7}
\|T^\ast_{2,R}g\|_{X_{N,\varrho_k}(I)}\lesssim \big(1+|k|(|\lambda-b(y_\ast)|^{1/2}+|\epsilon|^{1/2})\big)^{-1/4}\|g\|_{X_{N,\varrho_k}(I)}
\end{equation}

{\it Case II: $1/|k|\ll |\lambda-b(y_{\ast})|^{1/2}+|\epsilon|^{1/2}$.} In this case, we have for $|z-y_{\ast}|\lesssim\delta$ and $|y-y_{\ast}|\lesssim1$  that
\begin{equation}\label{DeC5.8}
\big|\mathcal{G}_{k}(y,z;\lambda+i\epsilon)\big|+|k|^{-1}\big|\partial_y\mathcal{G}_{k}(y,z;\lambda+i\epsilon)\big|\lesssim |k|^{-1}e^{-|k||y-z|}. 
\end{equation}
The desired bound 
\begin{equation}\label{DeC5.9}
\|T^\ast_{2,R}g\|_{X_{N,\varrho_k}(I)}\lesssim \big(1+|k|(|\lambda-b(y_\ast)|^{1/2}+|\epsilon|^{1/2})\big)^{-1/4}\|g\|_{X_{N,\varrho_k}(I)}
\end{equation}
follows from \eqref{DeC5.8}. 

{\bf Step 2.} We now turn to the proof of \eqref{DeC5.50} with $\diamond=S$ and still consider two cases. 

{\it Case I: $1/|k|>|\lambda-b(y_{\ast})|^{1/2}+|\epsilon|^{1/2}$}.
Denoting for $y\in I$, 
\begin{equation}\label{DeC5.90}
\varphi^\ast\big(\frac{y-y_{\ast}}{\delta}\big):=\Big[1-\varphi\big(\frac{z-y_{\ast}}{\delta'}\big)\Big]\varphi\big(\frac{z-y_{\ast}}{\delta}\big),
\end{equation}
we can rewrite 
\begin{equation}\label{DeC6.2}
\begin{split}
T^\ast_{2,S}g(y)&=\int_{0}^1\mathcal{G}_{k}(y,z;\lambda+i\epsilon)\varphi^\ast\big(\frac{z-y_{\ast}}{\delta}\big)\frac{b''(z)g(z)}{b'(z)}\partial_z\log{\frac{b(z)-\lambda+i\epsilon}{\delta^2}}\\
&=-\int_0^1\partial_z\bigg[\mathcal{G}_{k}(y,z;\lambda+i\epsilon)\varphi^\ast\big(\frac{z-y_{\ast}}{\delta}\big)\frac{b''(z)g(z)}{b'(z)}\bigg]\log{\frac{b(z)-\lambda+i\epsilon}{\delta^2}}dz.
\end{split}
\end{equation}
As a consequence of \eqref{DeC6.2}, we also have
\begin{equation}\label{DeC6.3}
\begin{split}
\partial_y\Big[T^\ast_{2,S}g(y)\Big]&=\partial_y\int_{0}^1\mathcal{G}_{k}(y,z;\lambda+i\epsilon)\varphi^\ast\big(\frac{z-y_{\ast}}{\delta}\big)\frac{b''(z)g(z)}{b'(z)}\partial_z\log{\frac{b(z)-\lambda+i\epsilon}{\delta^2}}dz\\
&=-\int_0^1\bigg[\partial_y(\partial_z+\partial_y)\mathcal{G}_{k}(y,z;\lambda,\epsilon)\varphi^\ast\big(\frac{z-y_{\ast}}{\delta}\big)\frac{b''(z)g(z)}{b'(z)}\bigg]\log{\frac{b(z)-\lambda+i\epsilon}{\delta^2}}dz\\
&\quad+\int_0^1\bigg[\partial^2_y\mathcal{G}_{k}(y,z;\lambda+i\epsilon)\varphi^\ast\big(\frac{z-y_{\ast}}{\delta}\big)\frac{b''(z)g(z)}{b'(z)}\bigg]\log{\frac{b(z)-\lambda+i\epsilon}{\delta^2}}dz\\
&\quad-\int_0^1\partial_y\mathcal{G}_{k}(y,z;\lambda+i\epsilon)\partial_z\bigg[\varphi^\ast\big(\frac{z-y_{\ast}}{\delta}\big)\frac{b''(z)g(z)}{b'(z)}\bigg]\log{\frac{b(z)-\lambda+i\epsilon}{\delta^2}}dz.
\end{split}
\end{equation}
Note that on the support of $\varphi^\ast(\frac{z-y_{\ast}}{\delta})$, we have 
\begin{equation}\label{DeC6.31}
|b'(z)|\approx \delta, \quad \varrho(z;\lambda+i\epsilon)\approx \delta.
\end{equation}
The desired bound \eqref{DeC5.50} for $\diamond=S$ follows from \eqref{DeC6.2}-\eqref{DeC6.3} and corollary \ref{mGk50}, and we have, in addition,
\begin{equation}\label{Dec6.32}
\begin{split}
&(\delta+|\epsilon|^{1/2})^{-1/2}\bigg\|\partial_y\bigg\{\partial_yT^\ast_{2,S}g(y)+\varphi^\ast\big(\frac{y-y_{\ast}}{\delta}\big)\frac{b''(y)g(y)}{b'(y)}\log{\frac{b(y)-\lambda+i\epsilon}{\delta^2}}\bigg\}\bigg\|_{L^2(S_{3(\delta+|\epsilon|^{1/2})})}\\
&\lesssim\delta^{-1/4} \Big[1+|k|(|\lambda-b(y_\ast)|^{1/2}+|\epsilon|^{1/2})\Big]^{-1/4}\|g\|_{X_{N,\varrho_k}(I)}.
\end{split}
\end{equation}

{\it Case II: $1/|k|\ll |\lambda-b(y_{\ast})|^{1/2}+|\epsilon|^{1/2}$.} This case is analogous to {\it Case I}, using Lemma \ref{mGk4} and Lemma \ref{mGk30}. 

Finally we turn to the assumption that $|\epsilon|^{1/2}\lesssim \delta$. Suppose $|\epsilon|^{1/2}\gg \delta$, then the factor $\frac{1}{b(z)-\lambda+i\epsilon}$ is not truly singular, and the desired bounds \eqref{DeC5.4} follow directly from the bounds on the modified Green's function $\mathcal{G}_{k}(y,z;\lambda+i\epsilon)$ from Lemma  \ref{mGk4} and Lemma \ref{mGk30}. Indeed, we have the stronger bound
\begin{equation}\label{DeC6.33}
\|T^\ast_{2}g\|_{X_{N,\varrho_k}(I)}\lesssim \frac{\delta}{\sqrt{|\epsilon|}}\|g\|_{X_{N,\varrho_k}(I)},
\end{equation}
which will be useful below. 
\end{proof}

The following limiting absorption principle plays an essential role in establishing the vorticity depletion phenomenon. 
\begin{lemma}\label{DeC7}
There exist positive numbers $\epsilon_0, \kappa$ such that the following statement holds. 

For $\epsilon\in[-\epsilon_0,\epsilon_0]\backslash\{0\}$, $\lambda\in \Sigma_{\delta_0}$,  $k\in\mathbb{Z}\backslash\{0\}$, and $X\in\{X_{N,\varrho_k}(I),X_{L,\varrho_k}(I)\}$,
\begin{equation}\label{DeC8}
\|(I+T^\ast_{k}(\lambda+i\epsilon))g\|_{X}\ge \kappa\| g\|_{X}, \quad {\rm for\,\,all}\,\,g\in H^{1}_k(I).
\end{equation}

\end{lemma}

\begin{proof}
We only consider the case $X=X_{N,\varrho_k}(I)$ as the other case is analogous. We prove \eqref{DeC8} by a contradiction argument.  Assume \eqref{DeC8} does not hold for any $\epsilon_0>0$. Then there exist for $\ell\in \Z\cap[1,\infty)$,
\begin{equation}\label{DeC9}
\lambda_\ell\to\lambda_\ast\in \Sigma_{\delta_0}, \,\,\epsilon_\ell\neq0 \,\,{\rm with}\,\,\epsilon_\ell\to0, \,\,k_\ell\to k_\ast\in (\Z\backslash\{0\})\cup\{\pm\infty\},
\end{equation}
and functions $g_\ell$ satisfying
\begin{equation}\label{DeC9.1}
\|g_\ell\|_{X_{N,\varrho_{k_\ell}}(I)}=1
\end{equation} 
such that 
\begin{equation}\label{DeC10}
\big\|(I+T^\ast_{k_\ell}(\lambda_\ell+i\epsilon_\ell))g_\ell\big\|_{X_{N,\varrho_{k_\ell}}(I)}\to 0.
\end{equation}
We can assume that $\lambda_\ast=b(y_\ast)$, otherwise the proof follows from the argument in the non-degenerate case. We consider several cases. 

{\it Case I: $\limsup_{\ell\to\infty}\|g_\ell\|_{H^1(I\backslash S_{\delta_0})}>0$.} By the bound \eqref{DeC5.3}, we can assume that $k_\ast\in \Z\backslash\{0\}$. By the bounds \eqref{DeC9.1} and \eqref{DeC10}, we can assume (passing to a subsequence if necessary) that
\begin{equation}\label{DeC11}
g_\ell\to g, \,\,{\rm in}\,\, H^1_{{\rm loc}}(I\backslash \{y_\ast\})\,\,{\rm as}\,\,\ell\to\infty, \quad g(0)=g(1)=0.\end{equation}
Then it follows from \eqref{DeC9.1} and \eqref{DeC10} that $g$ satisfies
\begin{equation}\label{DeC11.1}
|g(y)|\lesssim |y-y_{\ast}|^{7/4},
\end{equation}
and for $y\in (0,1)$, 
\begin{equation}\label{DeC12}
(k_\ast^2-\partial_y^2)g(y)+\frac{b''(y)}{b(y)-b(y_{\ast})}g(y)=0,
\end{equation}
which imply that $b(y_\ast)$ is an embedded eigenvalue for $L_k$, a contradiction to the spectral assumption. 

{\it Case II: $\limsup_{\ell\to\infty}\|g_\ell\|_{H^1(I\backslash S_{\delta_0})}=0$.} By the bound \eqref{DeC5} we can assume that $|k_\ell|(\delta_\ell+|\epsilon_\ell|^{1/2})\lesssim 1$. In this case,  using \eqref{DeC10}, we obtain that (passing to a subsequence if necessary)
\begin{equation}\label{DeC13}
\begin{split}
&\big\|(|\lambda_\ell-b(y_{\ast})|+|\epsilon|)^{-9/8}g_\ell\big\|_{L^2([y_{\ast}-\delta_\ell-|\epsilon_\ell|^{1/2}, \,y_{\ast}+\delta_\ell+|\epsilon_\ell|^{1/2}])}\\
&+\big\|(|\lambda_\ell-b(y_{\ast})|+|\epsilon|)^{-5/8}\partial_yg_\ell\big\|_{L^2([y_{\ast}-\delta_\ell-|\epsilon_\ell|^{1/2}, \,y_{\ast}+\delta_\ell+|\epsilon_\ell|^{1/2}])}\ge \sigma>0,
\end{split}
\end{equation}
where we recall from \eqref{mGk3} that
\begin{equation}\label{DeC14}
\delta_\ell\approx|\lambda_\ell-b(y_{\ast})|^{1/2}. 
\end{equation}

We divide into several subcases.

{\it Subcase II.1: $|\epsilon_\ell|^{1/2}\approx \delta_\ell$ for a subsequence.}

Define the change of variables for $\ell\ge1, y\in I$,
\begin{equation}\label{DeC15}
y-y_{\ast}=\delta_\ell Y, \quad g_\ell(y):=(|\lambda_\ell-b(y_{\ast})|+|\epsilon_\ell|)^{7/8}H_\ell(Y). 
\end{equation}
It follows from \eqref{Dec6.32} that we can extract a nontrivial limit $H\in H^1(\R)$ of $H_\ell$ satisfying for $Y\in \R$,
\begin{equation}\label{DeC16}
(\beta^2-\partial_Y^2)H(Y)+\frac{b''(y_{\ast})}{b''(y_{\ast})Y^2/2+\gamma+i\alpha}H(Y)=0,
\end{equation}
where $\beta\in\R, \alpha, \gamma\in\R\backslash\{0\}.$ This is impossible since the shear flow $(b''(y_\ast)Y^2/2,0), Y\in\R$ is spectrally stable, thanks to Rayleigh's inflection point criteria.

{\it Subcase II.2: $|\epsilon_\ell|^{1/2}=o( \delta_\ell)$ for a subsequence. } Passing to a subsequence and using rescaling as in \eqref{DeC15} we can extract a nontrivial limit $H\in H^1(\R)$, such that 
\begin{equation}\label{DeC17}
(\beta^2-\partial_Y^2)H(Y)+\lim_{\epsilon\to0}\frac{b''(y_{\ast})}{b''(y_{\ast})Y^2/2+\gamma+i\epsilon}H(Y)=0. 
\end{equation}
This is again impossible since the shear flow $(b''(y_\ast)Y^2/2,0), Y\in\R$ is spectrally stable. 

{\it Subcase II.3: $\delta_\ell=o(|\epsilon_\ell|^{1/2})$ for a subsequence. } This case is not possible thanks to the bound \eqref{DeC6.33}.
The lemma is now proved. 

\end{proof}

%
%
\section{Bounds on $\psi^\iota_{k,\epsilon}$: the non-degenerate case}
 In this section we obtain bounds on $\psi_{k,\epsilon}^\iota(y,\lambda)$  in the non-degenerate case, i.e. when $\lambda\in \Sigma\backslash\Sigma_{\delta_0/2}$.
 Since the arguments are analogous to those in \cite{JiaL}, we will be brief in the proofs, and provide only comments on the main ideas involved. 
 
 We begin with the following preliminary bounds.
 \begin{lemma}\label{bsdn1}
For $\lambda\in \Sigma\backslash\Sigma_{\delta_0/2},  k\in\mathbb{Z}\backslash\{0\}, \iota\in\{\pm\}$ and $0<\epsilon<\epsilon_0$, we have the bounds
\begin{equation}\label{bsdn2}
\|\psi_{k,\epsilon}^\iota(\cdot,\lambda)\|_{H^1_k(I)}\lesssim|k|^{-1/2}\|\omega_{0k}\|_{H^1_k(I)}.
\end{equation}
\end{lemma}
\begin{proof}
The desired bounds \eqref{bsdn2} follow directly from the Rayleigh equation \eqref{F7} and Lemma \ref{T5}, once we use the Green's function $G_k$ to invert $k^2-\partial_y^2$ and formulate \eqref{F7} as an integral equation.
\end{proof}

To obtain control on $\partial_\lambda \psi_{k,\epsilon}^\iota(\cdot,\lambda)$ for $\lambda\in \Sigma\backslash\Sigma_{\delta_0/2}$, we take derivative in \eqref{F7}, and obtain that 
\begin{equation}\label{bsdn2.5}
\begin{split}
(k^2-\partial_y^2)\partial_{\lambda}\psi^{\iota}_{k,\epsilon}(y,\lambda)+\frac{b''(y)\partial_{\lambda}\psi^{\iota}_{k,\epsilon}(y,\lambda)}{b(y)-\lambda+i\iota\epsilon}=\frac{\omega_0^k(y)}{(b(y)-\lambda+i\iota\epsilon)^2}-\frac{b''(y)\psi^{\iota}_{k,\epsilon}(z,\lambda)}{(b(y)-\lambda+i\iota\epsilon)^2},
\end{split}
\end{equation}
for $y\in I$ with zero boundary value at $y\in\{0,1\}$. 
 Reformulating \eqref{bsdn2.5} as an integral equation, we obtain that 
\begin{equation}\label{bsdn3}
\begin{split}
&\partial_{\lambda}\psi^{\iota}_{k,\epsilon}(y,\lambda)+\int_0^1G_k(y,z)\frac{b''(z)\partial_{\lambda}\psi^{\iota}_{k,\epsilon}(z,\lambda)}{b(z)-\lambda+i\iota\epsilon}\,dz\\
&=\int_0^1G_k(y,z)\frac{\omega_0^k(z)}{(b(z)-\lambda+i\iota\epsilon)^2}\,dz-\int_0^1G_k(y,z)\frac{b''(z)\psi^{\iota}_{k,\epsilon}(z,\lambda)}{(b(z)-\lambda+i\iota\epsilon)^2}\,dz.
\end{split}
\end{equation}
Recall the definition of the smooth cutoff function $\Phi$ below \eqref{T1}. We have the following bounds for $\partial_{\lambda}\psi^{\iota}_{k,\epsilon}(y,\lambda)$ when $\lambda\in \Sigma\backslash\Sigma_{\delta_0}$. 

\begin{lemma}\label{L0.2}
For $\lambda\in \Sigma\backslash\Sigma_{\delta_0/2},  k\in\mathbb{Z}\backslash\{0\}, \iota\in\{\pm\}$ and $0<\epsilon<\epsilon_0$, 
$\partial_{\lambda}\psi^{\iota}_{k,\epsilon}(y,\lambda)$ satisfies the following decomposition
\begin{equation}\label{L0.21}
\begin{split}
\partial_{\lambda}\psi^{\iota}_{k,\epsilon}(y,\lambda)=&\bigg[\frac{b'(y_0)\omega^k_0(y)}{|b'(y)|^2}-\frac{b''(y)\psi^{\iota}_{k,\epsilon}(y,\lambda)}{|b'(y)|^2}\bigg](1-\Phi(y))\log{(b(y)-\lambda+i\iota\epsilon)}\\  
&\\
         &+\sum_{\sigma=0,1}\omega_0^k(\sigma)\Psi^{\iota}_{k,\sigma,\epsilon}(y,\lambda)\log{(b(\sigma)-\lambda+i\iota\epsilon)}+\mathcal{R}^{\iota}_{\sigma,k,y_0,\epsilon}(y).
\end{split}
\end{equation}
In the above for $\sigma\in\{0,1\}$, $\iota\in\{\pm\}$, $0<\epsilon<\epsilon_0$, and $\lambda\in\Sigma\backslash\Sigma_{\delta_0/2}$,
\begin{equation}\label{L0.211}
\left\|\mathcal{R}^{\iota}_{\sigma,k,y_0,\epsilon}\right\|_{H^1_k(I)}\lesssim |k|^{1/2} \|\omega_{0k}\|_{H^2_k(I)}, \quad \left\|\Psi^{\iota}_{k,\sigma,\epsilon}(\cdot,\lambda)\right\|_{H^1_k(I)}\lesssim |k|^{-1/2}.
\end{equation}

\end{lemma}

\begin{proof}
The basic idea is to expand the right hand side of \eqref{bsdn3} using integration by parts, and apply Lemma \ref{T5} after removing the most singular parts. Indeed, denoting schematically, 
\begin{equation}\label{LLL0.1}
\mathcal{U}:=\int_0^1G_k(y,z)\frac{\omega_0^k(z)}{(b(z)-\lambda+i\iota\epsilon)^2}\,dz-\int_0^1G_k(y,z)\frac{b''(z)\psi^{\iota}_{k,\iota\epsilon}(z,\lambda)}{(b(z)-\lambda+i\iota\epsilon)^2}\,dz,
\end{equation}
we note that $\partial_{\lambda}\psi^{\iota}_{k,\epsilon}(y,\lambda)-\mathcal{U}$ satisfies the equation (recalling \eqref{T1} for the definition of $T_{k,\lambda,\iota\epsilon}$),
\begin{equation}\label{LLL0.2}
(I+T_{k,\lambda,\iota\epsilon})\big[\partial_{\lambda}\psi^{\iota}_{k,\epsilon}(y,\lambda)-\mathcal{U}\big]=-T_{k,\lambda,\iota\epsilon}\mathcal{U}.
\end{equation}
The term $T_{k,\lambda,\iota\epsilon}\mathcal{U}\in H^1_k(I)$ (noting however that for the boundary terms we need to track the singular coefficient $\log{(b(\sigma)-\lambda+i\iota\epsilon)}, \sigma\in\{0,1\}$), and we can apply Lemma \ref{T5} to \eqref{LLL0.2} in order to obtain the desired conclusions.  We refer to \cite{JiaL} for the detailed proof.
\end{proof}

To obtain bounds on $\partial^2_\lambda\psi_{k,\epsilon}^\iota(y,\lambda)$ for $\lambda\in \Sigma\backslash\Sigma_{\delta_0/2}$,  we take two derivatives in \eqref{F7} and obtain that 
\begin{equation}\label{bsdn4}
\begin{split}
&(k^2-\partial_y^2)\partial^2_{\lambda}\psi^{\iota}_{k,\epsilon}(y,\lambda)+\frac{b''(y)\partial^2_{\lambda}\psi^{\iota}_{k,\epsilon}(y,\lambda)}{b(y)-\lambda+i\iota\epsilon}\\
&=2\frac{\omega_0^k(y)}{(b(y)-\lambda+i\iota\epsilon)^3}-2\frac{b''(y)\psi^{\iota}_{k,\epsilon}(z,\lambda)}{(b(y)-\lambda+i\iota\epsilon)^3}+\frac{b''(y)\partial_\lambda\psi^{\iota}_{k,\epsilon}(z,\lambda)}{(b(y)-\lambda+i\iota\epsilon)^2},
\end{split}
\end{equation}
for $y\in I$ with zero boundary value at $y\in\{0,1\}$. 
We can reformulate \eqref{bsdn4} in the integral form for $y\in I$, as
\begin{equation}\label{bsdn5}
\begin{split}
&\partial^2_{\lambda}\psi^{\iota}_{k,\epsilon}(y,\lambda)+\int_0^1G_k(y,z)\frac{b''(z)\partial^2_{\lambda}\psi^{\iota}_{k,\epsilon}(z,\lambda)}{b(z)-\lambda+i\iota\epsilon}\,dz\\
&=\int_0^1G_k(y,z)\bigg[2\frac{\omega_0^k(z)}{(b(z)-\lambda+i\iota\epsilon)^3}-2\frac{b''(z)\psi^{\iota}_{k,\epsilon}(z,\lambda)}{(b(z)-\lambda+i\iota\epsilon)^3}+\frac{b''(z)\partial_\lambda\psi^{\iota}_{k,\epsilon}(z,\lambda)}{(b(z)-\lambda+i\iota\epsilon)^2}\bigg]\,dz.
\end{split}
\end{equation}

We have the following bounds on $\partial_\lambda^2\psi_{k,\epsilon}^\iota(y,\lambda)$ for $\lambda\in\Sigma\backslash\Sigma_{\delta_0/2}$. 
\begin{lemma}\label{L1}
For $k\in\mathbb{Z}\backslash\{0\}, \iota\in\{\pm\}$ and $0<\epsilon<\epsilon_0$, we have  the following bound
\begin{equation}\label{L2.2}
\begin{split}
&\bigg\|\partial_{\lambda}^2\psi^{\iota}_{k,\epsilon}(y,\lambda)-\frac{\omega_0^k(1)\Phi^{1\iota}_{k,\epsilon}(y,\lambda)}{b(1)-\lambda+i\iota\epsilon}-\frac{\omega_0^k(0)\Phi^{0\iota}_{k,\epsilon}(y,\lambda)}{b(0)-\lambda+i\iota\epsilon}-\frac{b''(y)\psi^{\iota}_{k,\epsilon}(y,\lambda)-\omega_0^k(y)}{|b'(y)|^2(b(y)-\lambda+i\iota\epsilon)}\bigg\|_{L^2(y\in I, \lambda\in \Sigma\backslash\Sigma_{\delta_0/2})}\\
&\lesssim |k|^{3/2}\|\omega_{0k}\|_{H^3_k(I)}\end{split}\end{equation}
In the above the functions $\Phi^{\sigma\iota}_{k,\epsilon}, \sigma\in\{0,1\}$ satisfy the equation for $y\in I$
\begin{equation}\label{L8.31'}
\begin{split}
(I+T_{k,\lambda,\iota\epsilon})\Phi^{1\iota}_{k,\epsilon}&=\frac{\sinh{(ky)}}{|b'(1)|^2\sinh{k}},\\
(I+T_{k,\lambda,\iota\epsilon})\Phi^{0\iota}_{k,\epsilon}&=\frac{\sinh{(k(1-y))}}{|b'(0)|^2\sinh{k}}.
\end{split}
\end{equation}

\end{lemma}

\begin{proof}
The main idea of the proof is to expand the right side of \eqref{bsdn5} and apply Lemma \ref{T5} after removing the most singular terms. Indeed, denoting schematically,
\begin{equation}\label{LLL2.0}
\mathcal{U}^\ast:=\int_0^1G_k(y,z)\bigg[2\frac{\omega_0^k(z)}{(b(z)-\lambda+i\iota\epsilon)^3}-2\frac{b''(z)\psi^{\iota}_{k,\iota\epsilon}(z,\lambda)}{(b(z)-\lambda+i\iota\epsilon)^3}+\frac{b''(z)\partial_\lambda\psi^{\iota}_{k,\iota\epsilon}(z,\lambda)}{(b(z)-\lambda+i\iota\epsilon)^2}\bigg]\,dz,
\end{equation}
we have 
\begin{equation}\label{LLL2.1}
(I+T_{k,\lambda,\iota\epsilon})\Big[\partial^2_{\lambda}\psi^{\iota}_{k,\epsilon}(y,\lambda)-\mathcal{U}^\ast+T_{k,\lambda,\iota\epsilon}\mathcal{U}^\ast\Big]=\big[T_{k,\lambda,\iota\epsilon}\big]^2\mathcal{U}^\ast.
\end{equation}
We note that $\partial^2_{\lambda}\psi^{\iota}_{k,\epsilon}(y,\lambda)-\mathcal{U}^\ast+T_{k,\lambda,\iota\epsilon}\mathcal{U}^\ast\in H^1_k(I)$ (however we again need to track the singularities in $\lambda$ in the boundary terms, involving $\log(b(\sigma)-\lambda+i\iota\epsilon)$ and $1/(b(\sigma)-\lambda+i\iota\epsilon)$ for $\sigma\in\{0,1\}$), and we can apply Lemma \eqref{T5} in order to obtain the desired conclusions. We refer to \cite{JiaL} for the detailed proof. 
\end{proof}

\section{Bounds on $\psi^\iota_{k,\epsilon}$: the degenerate case}
In this section we use the limiting absorption principle to study the Rayleigh equation \eqref{F7} for $\lambda\in \Sigma_{\delta_0}$. More precisely, write for $k\in\Z\backslash\{0\}, \iota\in\{\pm\}, \lambda\in \Sigma_{\delta_0}, 0<\epsilon<\epsilon_0$, (recall the definition of $\epsilon_0$ from Lemma \ref{DeC7})
\begin{equation}\label{BDC1}
\psi^\iota_{k,\epsilon}(y,\lambda)=\phi^\iota_{k,\epsilon}(y,\lambda)+\Psi(y)\frac{1}{b''(y)}\omega_{0k}(y),
\end{equation}
where $\Psi\in C_c^\infty(S_{3\delta_0})$ and $\Psi\equiv 1$ on $S_{2\delta_0}$. Recall that $S_d=[y_\ast-d,y_\ast+d]$ for $d>0$ from \eqref{LAP0.01}. Then $\phi^\iota_{k,\epsilon}(y,\lambda)$ satisfies for $y\in I$,
\begin{equation}\label{BDC2}
(k^2-\partial_y^2)\phi^\iota_{k,\epsilon}(y,\lambda)+\frac{b''(y)}{b(y)-\lambda+i\iota\epsilon}\phi^\iota_{k,\epsilon}(y,\lambda)=g^\iota_{k,\epsilon}(y,\lambda),\end{equation}
where for $k\in\Z\backslash\{0\}, \iota\in\{\pm\}, \lambda\in \Sigma_{\delta_0}, 0<\epsilon<\epsilon_0$
\begin{equation}\label{BDC3}
g^\iota_{k,\epsilon}(y,\lambda):=\frac{1-\Psi(y)}{b(y)-\lambda+i\iota\epsilon}\omega_{0k}(y)-(k^2-\partial_y^2)\Big[\frac{\Psi(y)}{b''(y)}\omega_{0k}(y)\Big].
\end{equation}

Our main results are bounds for the functions $\phi^\iota_{k,\epsilon}(y,\lambda)$. We begin with the following preliminary bounds.

\begin{lemma}\label{DCM1}
Assume that $k\in\Z\backslash\{0\}, \lambda\in \Sigma_{\delta_0}$ and let $\phi_{k,\epsilon}^{\iota}(y,\lambda)$ with $\iota\in\{\pm\}, \epsilon\in(0,\epsilon_0)$ be as defined in \eqref{BDC1}-\eqref{BDC2}. Recall from \eqref{mGk3} and \eqref{DeC1.92} that
\begin{equation}\label{DCM3.0}
\delta:=\delta(\lambda)=8\sqrt{|\lambda-b(y_\ast)|/|b''(y_\ast)|},\quad d_k=d_k(\lambda,\epsilon):=\big[|\lambda-b(y_\ast)|^{1/2}+|\epsilon|^{1/2}\big]\wedge \frac{1}{|k|}.
\end{equation}
 We have the bounds for $k\in\Z\backslash\{0\}, \epsilon\in(0,\epsilon_0), \iota\in\{\pm\}, \lambda\in \Sigma_{\delta_0}$,
\begin{equation}\label{DCM2}
\begin{split}
&\sum_{\alpha\in\{0,1\}}\big\|d_k^{-7/4+\alpha}\partial_y^{\alpha}\phi^\iota_{k,\epsilon}(y,\lambda)\big\|_{L^2\big([y_\ast-3(\delta+|\epsilon|^{1/2}),y_\ast+3(\delta+|\epsilon|^{1/2})]\big)}(\delta+|\epsilon|^{1/2})^{-1/2}\\
&+\sum_{\alpha\in\{0,1\}}\big\|(|y-y_\ast|\wedge d_k)^{-7/4+\alpha}\partial_y^\alpha\phi^\iota_{k,\epsilon}(y,\lambda)\big\|_{L^\infty\big([0,1]\backslash[y_\ast-3(\delta+|\epsilon|^{1/2}),y_\ast+3(\delta+|\epsilon|^{1/2})]\big)}\\
&\lesssim|k|^{5/2}\big\|\omega_{0k}\big\|_{H^3_k(I)}.
\end{split}
\end{equation}
Define for $y\in [0,1], k\in\Z\backslash\{0\}, \lambda\in \Sigma_{\delta_0}\backslash\{b(y_\ast)\}$,
\begin{equation}\label{DCM3}
\psi_k(y,\lambda):=\lim_{\epsilon\to0+}\Big[\psi_{k,\epsilon}^{+}(y,\lambda)-\psi_{k,\epsilon}^{-}(y,\lambda)\Big]=\lim_{\epsilon\to0+}\Big[\phi_{k,\epsilon}^{+}(y,\lambda)-\phi_{k,\epsilon}^{-}(y,\lambda)\Big].
\end{equation}
 Then we have the bounds for $\lambda \in  \Sigma_{\delta_0}\backslash\{b(y_\ast)\}$,
\begin{equation}\label{DCM4}
\begin{split}
&\sum_{\alpha\in\{0,1\}}\big\|(\delta\wedge |k|^{-1})^{-7/4+\alpha}\partial_y^{\alpha}\psi_k(y,\lambda)\big\|_{L^2([y_\ast-3\delta,y_\ast+3\delta])}\delta^{-1/2}\\
&+\sum_{\alpha\in\{0,1\}}\big\|(\delta\wedge |k|^{-1})^{-11/4}(|y-y_\ast|\wedge\frac{1}{|k|})^{1+\alpha}\partial_y^\alpha\psi_k(y,\lambda)\big\|_{L^\infty([0,1]\backslash[y_\ast-3\delta,y_\ast+3\delta]))}\\
&\lesssim|k|^{5/2}\big\|\omega_{0k}\big\|_{H^3_k(I)}.
\end{split}
\end{equation}

\end{lemma}

\begin{proof}
It follows from \eqref{BDC3} and our assumptions on the initial data $\omega_{0k}$ that we have the bound for $k\in\Z\backslash\{0\}, \iota\in\{\pm\}, 0<\epsilon<\epsilon_0$ and $\lambda\in\Sigma_{\delta_0}$,
\begin{equation}\label{BDC4}
\big\|g^\iota_{k,\epsilon}(y,\lambda)\big\|_{C(I)}\lesssim |k|^{5/2}\|\omega_{0k}\|_{H^3_k(I)}. 
\end{equation}
We can reformulate equation \eqref{BDC2} in the integral form as (recall the definition of $T^\ast(\lambda+i\epsilon)$ from \eqref{DeC3})
\begin{equation}\label{BDC5}
\phi^\iota_{k,\epsilon}(y,\lambda)+T^\ast_k(\lambda+i\iota\epsilon)\phi^\iota_{k,\epsilon}(y,\lambda)=\int_0^1\mathcal{G}_k(y,z;\lambda+i\iota\epsilon)g^\iota_{k,\epsilon}(z,\lambda)dz,
\end{equation}
for $y\in I$. By Lemma \ref{DeC7}, we obtain the bound
\begin{equation}\label{BDC6}
\big\|\phi^\iota_{k,\epsilon}(\cdot,\lambda)\big\|_{X_{N,\varrho_k}(I)}\lesssim\Big\|\int_0^1\mathcal{G}_k(y,z;\lambda+i\iota\epsilon)g^\iota_{k,\epsilon}(z,\lambda)dz\Big\|_{X_{N,\varrho_k}} \lesssim |k|^{5/2}\|\omega_{0k}\|_{H^3_k(I)},
\end{equation}
which, by the definition of the space $X_{N,\varrho_k}$, see \eqref{DeC2}, implies the desired bounds \eqref{DCM2}. 

For applications below on isolating the singularity at $\lambda=b(y)$, we fix $\varphi_\delta(y)\in C_c^\infty(S_{2\delta})$ as
\begin{equation}\label{BDC6.00001}
\varphi_\delta(y):=\varphi(\frac{y}{\delta})\big[1-\varphi(\frac{y}{\delta'})\big],
\end{equation}
for $y\in I$, with $\delta':=\delta/M$ and an $M\gg1$ sufficiently large such that $|b(y)-\lambda|\approx |\lambda-b(y_\ast)|$ for $|y-y_\ast|<\delta/M$. 

To prove \eqref{DCM4}, we note from \eqref{BDC2} that $\phi^+_{k,\epsilon}(y,\lambda)-\phi^-_{k,\epsilon}(y,\lambda)$ satisfies the equation for $y\in I$.
\begin{equation}\label{BDC6.01}
\begin{split}
&(k^2-\partial_y^2)\big[\phi^+_{k,\epsilon}(y,\lambda)-\phi^-_{k,\epsilon}(y,\lambda)\big]+\frac{b''(y)}{b(y)-\lambda+i\epsilon}\big[\phi^+_{k,\epsilon}(y,\lambda)-\phi^-_{k,\epsilon}(y,\lambda)\big]\\
&=g^+_{k,\epsilon}(y,\lambda)-g^-_{k,\epsilon}(y,\lambda)-\Big[\frac{b''(y)}{b(y)-\lambda+i\epsilon}-\frac{b''(y)}{b(y)-\lambda-i\epsilon}\Big]\phi^-_{k,\epsilon}(y,\lambda).
\end{split}
\end{equation}
Denote for $\lambda\in \Sigma_{\delta_0}\backslash\{b(y_\ast)\}$, $\epsilon\in(0,\epsilon_0)$ and $y\in I$ the function $h_{k,\epsilon}(y,\lambda)$ as the solution to 
\begin{equation}\label{DBC6.02}
\begin{split}
&(k^2-\partial_y^2)h_{k,\epsilon}(y,\lambda)+\frac{b''(y)}{b(y)-\lambda+i\epsilon}h_{k,\epsilon}(y,\lambda)=\varphi_\delta(y)\Big[\frac{b''(y)}{b(y)-\lambda-i\epsilon}-\frac{b''(y)}{b(y)-\lambda+i\epsilon}\Big]\phi^-_{k,\epsilon}(y,\lambda),
\end{split}
\end{equation}
with zero Dirichlet boundary condition. 
Then it is clear that for $\lambda\in \Sigma_{\delta_0}\backslash\{b(y_\ast)\}, y\in I,$
\begin{equation}\label{DBC6.03}
\psi_k(y,\lambda)=\lim_{\epsilon\to 0+}h_{k,\epsilon}(y,\lambda). 
\end{equation}
We can reformulate \eqref{DBC6.02} as the following integral equation for $\lambda\in \Sigma_{\delta_0}\backslash\{b(y_\ast)\}, y\in I,$
\begin{equation}\label{DBC6.04}
\begin{split}
&h_{k,\epsilon}(y,\lambda)+T^\ast_k(\lambda+i\epsilon)h_{k,\epsilon}(y,\lambda)\\
&=-\int_0^1\mathcal{G}_k(y,z;\lambda+i\epsilon)\varphi_\delta(z)\Big[\frac{b''(z)}{b(z)-\lambda+i\epsilon}-\frac{b''(z)}{b(z)-\lambda-i\epsilon}\Big]\phi^-_{k,\epsilon}(z,\lambda)\,dz.
\end{split}
\end{equation}
It follows from the bound \eqref{DCM2} that for $|\epsilon|\lesssim (\delta\wedge \frac{1}{|k|})^4$,
\begin{equation}\label{DBC6.05}
\bigg\|\int_0^1\mathcal{G}_k(y,z;\lambda+i\epsilon)\varphi_\delta(z)\Big[\frac{b''(z)}{b(z)-\lambda+i\epsilon}-\frac{b''(z)}{b(z)-\lambda-i\epsilon}\Big]\phi^-_{k,\epsilon}(z,\lambda)\,dz\bigg\|_{X_{L,\varrho_k}}\lesssim (\delta\wedge \frac{1}{|k|})^{7/4}. 
\end{equation}
The desired bound \eqref{DCM4} then follows from Lemma \ref{DeC7} with $X=X_{L, \varrho_k}$.

\end{proof}

To obtain higher order regularity bounds (in $\lambda$) of $\phi^\iota_{k,\epsilon}(\cdot,\lambda)$, we take the derivative $\partial_\lambda$ in \eqref{BDC2}. It follows that
$\partial_\lambda\phi^\iota_{k,\epsilon}(y,\lambda)$
satisfies for $y\in I$,
\begin{equation}\label{BDC8}
\begin{split}
&\Big[k^2-\partial_y^2+\frac{b''(y)}{b(y)-\lambda+i\iota\epsilon}\Big]\partial_\lambda\phi^\iota_{k,\epsilon}(y,\lambda)=-\frac{b''(y)}{(b(y)-\lambda+i\iota\epsilon)^2}\phi^\iota_{k,\epsilon}(y,\lambda)+\partial_\lambda g^\iota_{k,\epsilon}(y,\lambda),
\end{split}
\end{equation}
with zero Dirichlet boundary condition.

Recall the definition of $\varphi_\delta$ from \eqref{BDC6.00001}. We have the following bounds on $\partial_\lambda\phi^\iota_{k,\epsilon}(y,\lambda)$. 
\begin{lemma}\label{DCM5}
Assume that $k\in\Z\backslash\{0\}, \lambda\in \Sigma_{\delta_0}\backslash\{b(y_\ast)\}$. Let $\psi_{k,\epsilon}^{\iota}(y,\lambda)$ and $\phi_{k,\epsilon}^{\iota}(y,\lambda)$ with $\iota\in\{\pm\}, 0<\epsilon<\min\{|\lambda-b(y_\ast)|,\epsilon_0\}$ be as defined in \eqref{F7} and \eqref{BDC1} respectively. Recall from \eqref{mGk3} that
\begin{equation}\label{DCM6}
\delta:=\delta(\lambda)=8\sqrt{|\lambda-b(y_\ast)|/b''(y_\ast)}.
\end{equation}
Denote for $y\in [0,1], \iota\in\{\pm\}, \lambda\in \Sigma_{\delta_0}\backslash\{b(y_\ast)\}, 0<\epsilon<\min\{|\lambda-b(y_\ast)|,\epsilon_0\}$,
\begin{equation}\label{DCM7.0}
\begin{split}
&\Lambda^\iota_{1,\epsilon}(y,\lambda):=\,\phi^\iota_{k,\epsilon}(y,\lambda)\varphi_\delta(y)\frac{b''(y)}{(b'(y))^2}\log\frac{b(y)-\lambda+i\iota\epsilon}{\delta^2},\\
&\Lambda_1(y,\lambda):=\,\psi_k(y,\lambda)\varphi_\delta(y)\frac{b''(y)}{(b'(y))^2}\log\frac{b(y)-\lambda}{\delta^2}.
\end{split}
\end{equation}

 We have the bounds for $0<\epsilon<\min\{|\lambda-b(y_\ast)|,\epsilon_0\}, \iota\in\{\pm\}$, and $\lambda\in \Sigma_{\delta_0}$ that 
\begin{equation}\label{DCM7}
\begin{split}
&\sum_{\alpha\in\{0,1\}}\Big\|(\delta\wedge |k|^{-1})^{1/4+\alpha}\partial_y^{\alpha}\Big[\partial_\lambda\phi^\iota_{k,\epsilon}(y,\lambda)-\Lambda^\iota_{1,\epsilon}(y,\lambda)\Big]\Big\|_{L^2([y_\ast-3\delta,y_\ast+3\delta])}\delta^{-1/2}\\
&+\sum_{\alpha\in\{0,1\}}\Big\|(\delta\wedge |k|^{-1})^2(|y-y_\ast|\wedge\frac{1}{|k|})^{-7/4+\alpha}\partial_y^\alpha\partial_\lambda\phi^\iota_{k,\epsilon}(y,\lambda)\Big\|_{L^\infty([0,1]\backslash[y_\ast-3\delta,y_\ast+3\delta]))}\\
&\lesssim|k|^{5/2}\big\|\omega_{0k}\big\|_{H^3_k(I)}.
\end{split}
\end{equation}
In addition, we have the bounds for $\lambda \in  \Sigma_{\delta_0}\backslash\{b(y_\ast)\}$ and $k\in\Z\backslash\{0\}$,
\begin{equation}\label{DCM9}
\begin{split}
&\sum_{\alpha\in\{0,1\}}\big\|(\delta\wedge |k|^{-1})^{1/4+\alpha}\partial_y^{\alpha}\Big[\partial_\lambda\psi_k(y,\lambda)-\Lambda_1(y,\lambda)\Big]\Big\|_{L^2([y_\ast-3\delta,y_\ast+3\delta])}\delta^{-1/2}\\
&+\sum_{\alpha\in\{0,1\}}\big\|(\delta\wedge |k|^{-1})^{-3/4}(|y-y_\ast|\wedge\frac{1}{|k|})^{1+\alpha}\partial_y^\alpha\partial_\lambda\psi_k(y,\lambda)\big\|_{L^\infty([0,1]\backslash[y_\ast-3\delta,y_\ast+3\delta]))}\\
&\lesssim|k|^{5/2}\big\|\omega_{0k}\big\|_{H^3_k(I)}.
\end{split}
\end{equation}

\end{lemma}

\begin{proof}
Define for $k\in\Z\backslash\{0\}, \iota\in\{\pm\}, \lambda\in \Sigma_{\delta_0}\backslash\{b(y_\ast)\}, 0<\epsilon<\min\{|\lambda-b(y_\ast)|,\epsilon_0\}, y\in I$,
\begin{equation}\label{BDC10}
\partial_\lambda\phi^\iota_{k,\epsilon}(y,\lambda):=\phi^\iota_{k,\epsilon}(y,\lambda;1)+\int_0^1\mathcal{G}_k(y,z;\lambda+i\iota\epsilon)\Big[\frac{-b''(z)}{(b(z)-\lambda+i\iota\epsilon)^2}\phi^\iota_{k,\epsilon}(z,\lambda)+\partial_\lambda g^\iota_{k,\epsilon}(z,\lambda)\Big]\,dz.
\end{equation}
It follows from \eqref{BDC8} that $\phi^\iota_{k,\epsilon}(y,\lambda;1)$ satisfies for $y\in I$,
\begin{equation}\label{BDC12}
\begin{split}
&\phi^\iota_{k,\epsilon}(y,\lambda;1)+T^\ast_k(\lambda+i\iota\epsilon)\phi^\iota_{k,\epsilon}(y,\lambda;1)\\
&=-T^\ast_k(\lambda+i\iota\epsilon)\int_0^1\mathcal{G}_k(y,z;\lambda+i\iota\epsilon)\Big[-\frac{b''(z)}{(b(z)-\lambda+i\iota\epsilon)^2}\phi^\iota_{k,\epsilon}(z,\lambda)+\partial_\lambda g^\iota_{k,\epsilon}(z,\lambda)\Big]\,dz.
\end{split}
\end{equation}
Denote for $k\in\Z\backslash\{0\}, \iota\in\{\pm\}, \lambda\in \Sigma_{\delta_0}\backslash\{b(y_\ast)\}, 0<\epsilon<\min\{|\lambda-b(y_\ast)|,\epsilon_0\}, z\in I$,
\begin{equation}\label{BDC12.01}
\begin{split}
h_{k,\epsilon}^\iota(z,\lambda;1):=&\frac{b''(z)}{(b(z)-\lambda+i\iota\epsilon)^2}\varphi_\delta(z)\phi^\iota_{k,\epsilon}(z,\lambda),\\
h_{k,\epsilon}^\iota(z,\lambda;2):=&\frac{b''(z)}{(b(z)-\lambda+i\iota\epsilon)^2}(1-\varphi_\delta(z))\phi^\iota_{k,\epsilon}(z,\lambda),\quad h_{k,\epsilon}^\iota(z,\lambda;3):=\partial_\lambda g^\iota_{k,\epsilon}(z,\lambda).
\end{split}
\end{equation}
It follows from the bound \eqref{DCM2} and Lemma \ref{mGk4} that for $j\in\{2,3\}$
\begin{equation}\label{BDC12.02}
\big\|T^\ast_k(\lambda+i\iota\epsilon)\int_0^1\mathcal{G}_k(y,z;\lambda+i\iota\epsilon)h_{k,\epsilon}^\iota(z,\lambda;j)\,dz\big\|_{X_{N,\varrho_k}}\lesssim (\delta\wedge |k|^{-1})^{-2}|k|^{5/2}\|\omega_{0k}\|_{H^3_k(I)}.
\end{equation}
Using integration by parts argument similar to \eqref{DeC6.2}-\eqref{DeC6.3}, we have also 
\begin{equation}\label{BDC12.03}
\begin{split}
&\bigg\|T^\ast_k(\lambda+i\iota\epsilon)\int_0^1\mathcal{G}_k(y,z;\lambda+i\iota\epsilon)h_{k,\epsilon}^\iota(z,\lambda;1)\,dz\bigg\|_{X_{N,\varrho_k}}\lesssim (\delta\wedge |k|^{-1})^{-2}|k|^{5/2}\big\|\omega_{0k}\big\|_{H^3_k(I)}.
\end{split}
\end{equation}
It follows from \eqref{BDC12.02}-\eqref{BDC12.03} and Lemma \ref{DeC7} that for $\lambda\backslash\{b(y_\ast)\}$,
\begin{equation}\label{BDC12.04}
\big\|\phi^\iota_{k,\epsilon}(y,\lambda;1)\big\|_{X_{N,\varrho_k}}\lesssim (\delta\wedge |k|^{-1})^{-2}|k|^{5/2}\big\|\omega_{0k}\big\|_{H^3_k(I)}.
\end{equation}
The desired bound \eqref{DCM7} follows, as a consequence of \eqref{BDC12.04} and \eqref{BDC10}. 

Using \eqref{BDC8}, we get that for $y\in I$, 
\begin{equation}\label{BDC12.05}
\begin{split}
&\Big[k^2-\partial_y^2+\frac{b''(y)}{b(y)-\lambda+i\epsilon}\Big]\big[\partial_\lambda\phi^+_{k,\epsilon}(y,\lambda)-\partial_\lambda\phi^-_{k,\epsilon}(y,\lambda)\big]\\
&=-\bigg[\frac{b''(y)}{(b(y)-\lambda+i\epsilon)^2}\phi^+_{k,\epsilon}(y,\lambda)-\frac{b''(y)}{(b(y)-\lambda-i\epsilon)^2}\phi^-_{k,\epsilon}(y,\lambda)\bigg]+\big[\partial_\lambda g^+_{k,\epsilon}(y,\lambda)-\partial_\lambda g^-_{k,\epsilon}(y,\lambda)\big]\\
&\quad-\bigg[\frac{b''(y)}{b(y)-\lambda+i\epsilon}-\frac{b''(y)}{b(y)-\lambda-i\epsilon}\bigg]\partial_\lambda\phi^-_{k,\epsilon}(y,\lambda),
\end{split}
\end{equation}
with zero Dirichlet boundary condition. 

Denoting for $\lambda\in\Sigma_{\delta_0}\backslash\{b(y_\ast)\}$ and $y\in I$, $D\phi_{k,\epsilon}(y,\lambda)$ as the solution to 
\begin{equation}\label{BDC12.06}
\begin{split}
&\Big[k^2-\partial_y^2+\frac{b''(y)}{b(y)-\lambda+i\iota\epsilon}\Big]D\phi_{k,\epsilon}(y,\lambda)\\
&=-\varphi_\delta(y)\bigg[\frac{b''(y)}{(b(y)-\lambda+i\epsilon)^2}\phi^+_{k,\epsilon}(y,\lambda)-\frac{b''(y)}{(b(y)-\lambda-i\epsilon)^2}\phi^-_{k,\epsilon}(y,\lambda)\bigg]\\
&\quad-\varphi_\delta(y)\bigg[\frac{b''(y)}{b(y)-\lambda+i\iota\epsilon}-\frac{b''(y)}{b(y)-\lambda-i\iota\epsilon}\bigg]\partial_\lambda\phi^-_{k,\epsilon}(y,\lambda),
\end{split}
\end{equation}
for $y\in I$ with zero Dirichlet boundary condition.

We notice the identity that for $y\in I, \lambda\in\Sigma_{\delta_0}\backslash\{b(y_\ast)\}$,
\begin{equation}\label{BDC12.061}
\partial_\lambda\psi_k(y,\lambda)=\lim_{\epsilon\to 0+}D\phi_{k,\epsilon}(y,\lambda). 
\end{equation}
We can reformulate \eqref{BDC12.06} as the integral equation for $y\in I$,
\begin{equation}\label{BDC12.07}
\begin{split}
&D\phi_{k,\epsilon}(y,\lambda)+T^\ast_k(\lambda+i\epsilon)D\phi_{k,\epsilon}(y,\lambda)\\
&=-\int_0^1\mathcal{G}_k(y,z;\lambda+i\epsilon)\varphi_\delta(z)\bigg[\frac{b''(z)}{(b(z)-\lambda+i\epsilon)^2}\phi^+_{k,\epsilon}(z,\lambda)-\frac{b''(z)}{(b(z)-\lambda-i\epsilon)^2}\phi^-_{k,\epsilon}(z,\lambda)\bigg]\,dz\\
&\quad-\int_0^1\mathcal{G}_k(y,z;\lambda+i\epsilon)\varphi_\delta(z)\bigg[\frac{b''(z)}{b(z)-\lambda+i\epsilon}-\frac{b''(z)}{b(z)-\lambda-i\iota\epsilon}\bigg]\partial_\lambda\phi^-_{k,\epsilon}(z,\lambda)\,dz\\
&:=R_{k,\epsilon}(y,\lambda).
\end{split}
\end{equation}
We can write for $y\in I, \lambda\in\Sigma_{\delta_0}\backslash\{b(y_\ast)\}, 0<\epsilon<\min\{|\lambda-b(y_\ast)|,\epsilon_0\}$,
\begin{equation}\label{BDC30.1}
D\phi_{k,\epsilon}(y,\lambda):=R_{k,\epsilon}(y,\lambda)+D\phi_{k,\epsilon}(y,\lambda;1).
\end{equation}
Then $D\phi_{k,\epsilon}(y,\lambda;1)$ satisfies for $y\in I, \lambda\in\Sigma_{\delta_0}\backslash\{b(y_\ast)\}, 0<\epsilon<\min\{|\lambda-b(y_\ast)|,\epsilon_0\}$, the equation
\begin{equation}\label{BDC30.2}
D\phi_{k,\epsilon}(y,\lambda;1)+T_k^\ast(\lambda+i\epsilon)D\phi_{k,\epsilon}(y,\lambda;1)=-T_k^\ast(\lambda+i\epsilon)R_{k,\epsilon}(y,\lambda).
\end{equation}
The desired bounds \eqref{DCM9} follow from \eqref{BDC12.07}-\eqref{BDC30.2}, and Lemma \ref{mGk30} with $X=X_{L,\varrho_k}$. 
\end{proof}

Lastly we turn to the highest order derivative $\partial_\lambda^2\psi_{k,\epsilon}^\iota(y,\lambda)$ that we need to control. To study $\partial_\lambda^2\psi_{k,\epsilon}^\iota(y,\lambda)$, we take the derivative $\partial_\lambda$ in \eqref{BDC8} and obtain that
\begin{equation}\label{BDC8.00}
\begin{split}
\Big[k^2-\partial_y^2+\frac{b''(y)}{b(y)-\lambda+i\iota\epsilon}\Big]\partial_\lambda^2\phi^\iota_{k,\epsilon}(\cdot,\lambda)=&-\frac{2b''(y)}{(b(y)-\lambda+i\iota\epsilon)^2}\partial_\lambda\phi^\iota_{k,\epsilon}(\cdot,\lambda)\\
&-\frac{2b''(y)}{(b(y)-\lambda+i\iota\epsilon)^3}\phi^\iota_{k,\epsilon}(y,\lambda)+\partial^2_\lambda g^\iota_{k,\epsilon}(y,\lambda).
\end{split}
\end{equation}

\begin{lemma}\label{DCM100}
Assume that $k\in\Z\backslash\{0\}, \lambda\in \Lambda_{\delta_0}\backslash\{b(y_\ast)\}$ and let $\phi_{k,\epsilon}^{\iota}(y,\lambda)$ with $\iota\in\{\pm\}, 0<\epsilon<\min\{|\lambda-b(y_\ast)|,\epsilon_0\}$ be as defined in \eqref{BDC2}. Recall that
\begin{equation}\label{DCM11}
\delta:=\delta(\lambda)=8\sqrt{|\lambda-b(y_\ast)|/b''(y_\ast)}.
\end{equation}
Denoting for $y\in [0,1], \lambda\in \Lambda_{\delta_0}\backslash\{b(y_\ast)\}$,
\begin{equation}\label{DCM7.1}
\begin{split}
\Lambda_2(y,\lambda):=&-\,\psi_k(y,\lambda)\varphi_\delta(y)\frac{b''(y)}{(b'(y))^2}\lim_{\epsilon\to0+}\frac{1}{b(y)-\lambda+i\epsilon}\\
&-\varphi_\delta(y)\frac{b''(y)}{(b'(y))^2}\lim_{\epsilon\to0+}\Big[\frac{1}{b(y)-\lambda+i\epsilon}-\frac{1}{b(y)-\lambda-i\epsilon}\Big]\phi^-_{k,\epsilon}(y,\lambda),
\end{split}
\end{equation}
then we have the bounds for $\lambda \in   \Lambda_{\delta_0}\backslash\{b(y_\ast)\}$,
\begin{equation}\label{DCM9}
\begin{split}
&\sum_{\alpha\in\{0,1\}}\Big\|(\delta\wedge |k|^{-1})^{9/4}\Big[\partial_\lambda^2\psi_k(y,\lambda)-\Lambda_2(y,\lambda)\Big]\Big\|_{L^2([y_\ast-3\delta,y_\ast+3\delta])}\delta^{-1/2}\\
&+\sum_{\alpha\in\{0,1\}}\Big\|(\delta\wedge |k|^{-1})^{5/4}(|y-y_\ast|\wedge\frac{1}{|k|})\partial_\lambda^2\psi_k(y,\lambda)\Big\|_{L^\infty([0,1]\backslash[y_\ast-3\delta,y_\ast+3\delta]))}\lesssim|k|^{5/2}\big\|\omega_{0k}\big\|_{H^3_k(I)}.
\end{split}
\end{equation}

\end{lemma}

\begin{proof}
Denote for $k\in\Z\backslash\{0\}, \lambda\in \Lambda_{\delta_0}\backslash\{b(y_\ast)\}$, $\iota\in\{\pm\}, 0<\epsilon<\min\{|\lambda-b(y_\ast)|,\epsilon_0\}$ and $y\in I$,
\begin{equation}\label{DCM12}
\begin{split}
h_{k,\epsilon}^\iota(z,\lambda;4):=&-\frac{2b''(z)}{(b(z)-\lambda-i\iota\epsilon)^2}\varphi_\delta(z)\partial_\lambda\phi^\iota_{k,\epsilon}(z,\lambda),\\
h_{k,\epsilon}^\iota(z,\lambda;5)=&-\frac{2b''(z)}{(b(z)-\lambda-i\iota\epsilon)^3}\varphi_\delta(z)\phi^\iota_{k,\epsilon}(z,\lambda)\\
 h_{k,\epsilon}^\iota(z,\lambda;6):=&-\frac{b''(z)}{(b(z)-\lambda-i\iota\epsilon)^2}(1-\varphi_\delta(z))\partial_\lambda\phi^\iota_{k,\epsilon}(z,\lambda),\\
h_{k,\epsilon}^\iota(z,\lambda;7)=&-\frac{2b''(z)}{(b(z)-\lambda-i\iota\epsilon)^3}(1-\varphi_\delta(z))\phi^\iota_{k,\epsilon}(z,\lambda) ,\quad h_{k,\epsilon}^\iota(z,\lambda;8):=\partial^2_\lambda g^\iota_{k,\epsilon}(z,\lambda).
\end{split}
\end{equation}
Define for $k\in\Z\backslash\{0\}, \lambda\in \Lambda_{\delta_0}\backslash\{b(y_\ast)\}$, $\iota\in\{\pm\}, 0<\epsilon<\min\{|\lambda-b(y_\ast)|,\epsilon_0\}$ and $z\in I$,
\begin{equation}\label{DCM10}
\begin{split}
\partial^2_\lambda\phi^\iota_{k,\epsilon}(y,\lambda):=&\,\phi^\iota_{k,\epsilon}(y,\lambda;2)+\sum_{j=4}^8\int_0^1\mathcal{G}_k(y,z;\lambda+i\iota\epsilon)h_{k,\epsilon}^\iota(z,\lambda;j)\,dz\\
&-\sum_{j=4}^8T^\ast_k(\lambda+i\iota\epsilon)\int_0^1\mathcal{G}_k(y,z;\lambda+i\iota\epsilon)h_{k,\epsilon}^\iota(z,\lambda;j)\,dz\\
\end{split}
\end{equation}
It follows from \eqref{BDC8.00} that $\phi^\iota_{k,\epsilon}(y,\lambda;2)$ satisfies for $y\in I$,
\begin{equation}\label{DCM11}
\begin{split}
&\phi^\iota_{k,\epsilon}(y,\lambda;2)+T^\ast_k(\lambda+i\iota\epsilon)\phi^\iota_{k,\epsilon}(y,\lambda;2)=\sum_{j=4}^8\big[T^\ast_k(\lambda+i\iota\epsilon)\big]^2\int_0^1\mathcal{G}_k(y,z;\lambda+i\iota\epsilon)h_{k,\epsilon}^\iota(z,\lambda;j)\,dz.
\end{split}
\end{equation}

It follows from Lemma \ref{DCM5} and Lemma \ref{mGk4} that for $j\in\{6,7,8\}$
\begin{equation}\label{DCM13}
\Big\|\big[T^\ast_k(\lambda+i\epsilon)\big]^2\int_0^1\mathcal{G}_k(y,z;\lambda+i\iota\epsilon)h_{k,\epsilon}^\iota(z,\lambda;j)\,dz\Big\|_{X_{N,\varrho_k}}\lesssim (\delta\wedge |k|^{-1})^{-4}|k|^{5/2}\|\omega_{0k}\|_{H^3_k(I)}.
\end{equation}
Using integration by parts argument similar to \eqref{DeC6.2}-\eqref{DeC6.3}, we have also for $j\in\{4,5\}$,
\begin{equation}\label{DCM14}
\begin{split}
&\bigg\|\big[T^\ast_k(\lambda+i\epsilon)\big]^2\int_0^1\mathcal{G}_k(y,z;\lambda+i\iota\epsilon)h_{k,\epsilon}^\iota(z,\lambda;j)\,dz\bigg\|_{X_{N,\varrho_k}}\lesssim (\delta\wedge |k|^{-1})^{-4}|k|^{5/2}\big\|\omega_{0k}\big\|_{H^3_k(I)}.
\end{split}
\end{equation}
It follows from \eqref{DCM12}-\eqref{DCM14} and Lemma \ref{DeC7} that for $\lambda\in \Lambda_{\delta_0}\backslash\{b(y_\ast)\}$, $\iota\in\{\pm\}, 0<\epsilon<\min\{|\lambda-b(y_\ast)|,\epsilon_0\},$
\begin{equation}\label{DCM15}
\big\|\phi^\iota_{k,\epsilon}(y,\lambda;2)\big\|_{X_{N,\varrho_k}}\lesssim (\delta\wedge |k|^{-1})^{-4}|k|^{5/2}\big\|\omega_{0k}\big\|_{H^1_k(I)}.
\end{equation}

Using \eqref{BDC8.00}, we get that for $y\in I$,
\begin{equation}\label{DCM16}
\begin{split}
&\bigg[k^2-\partial_y^2+\frac{b''(y)}{b(y)-\lambda+i\epsilon}\bigg]\big(\partial^2_\lambda\phi^+_{k,\epsilon}(y,\lambda)-\partial^2_\lambda\phi^-_{k,\epsilon}(y,\lambda)\big)=\sum_{j=4}^8\Big[h_{k,\epsilon}^+(y,\lambda;j)-h_{k,\epsilon}^-(y,\lambda;j)\Big].
\end{split}
\end{equation}
Denoting $D^2\phi_{k,\epsilon}(y,\lambda)$, $h\in I, \lambda\in \Lambda_{\delta_0}\backslash\{b(y_\ast)\}$, as the solution to 
\begin{equation}\label{DCM17}
\begin{split}
&\Big[k^2-\partial_y^2+\frac{b''(y)}{b(y)-\lambda+i\iota\epsilon}\Big]D^2\phi_{k,\epsilon}(y,\lambda)=\sum_{j=4}^5\Big[h_{k,\epsilon}^+(y,\lambda;j)-h_{k,\epsilon}^-(y,\lambda;j)\Big],
\end{split}
\end{equation}
for $y\in I$ with zero Dirichlet boundary condition.

We note the identity that for $y\in I, \lambda\in\Sigma_{\delta_0}\backslash\{b(y_\ast)\}$,
\begin{equation}\label{DCM18}
\partial^2_\lambda\psi_k(y,\lambda)=\lim_{\epsilon\to 0+}D^2\phi_{k,\epsilon}(y,\lambda). 
\end{equation}
We can reformulate \eqref{DCM17} as the integral equation for $y\in I$
\begin{equation}\label{DCM19}
\begin{split}
&D^2\phi_{k,\epsilon}(y,\lambda)+T^\ast_k(\lambda+i\epsilon)D^2\phi_{k,\epsilon}(y,\lambda)\\
&=\int_0^1\mathcal{G}_k(y,z;\lambda+i\epsilon)\varphi_\delta(z)\sum_{j=4}^5\Big[h_{k,\epsilon}^+(z,\lambda;j)-h_{k,\epsilon}^-(z,\lambda;j)\Big]\,dz:=R^\ast_{k,\epsilon}(y,\lambda).
\end{split}
\end{equation}
We can write for $\lambda\in\Sigma_{\delta_0}\backslash\{b(y_\ast)\}, 0<\epsilon<\min\{|\lambda-b(y_\ast)|,\epsilon_0\}, y\in I,$
\begin{equation}\label{DCM20}
D^2\phi_{k,\epsilon}(y,\lambda):=D^2\phi_{k,\epsilon}(y,\lambda;2)+R^\ast_{k,\epsilon}(y,\lambda)-T_k^\ast(\lambda+i\epsilon)R^\ast_{k,\epsilon}(y,\lambda).
\end{equation}
Then $D^2\phi_{k,\epsilon}(y,\lambda;2)$ satisfies for $y\in I, \lambda\in\Sigma_{\delta_0}\backslash\{b(y_\ast)\}$,
\begin{equation}\label{DCM21}
D^2\phi_{k,\epsilon}(y,\lambda;2)+T_k^\ast(\lambda+i\epsilon)D^2\phi_{k,\epsilon}(y,\lambda;2)=\big[T_k^\ast(\lambda+i\epsilon)\big]^2R^\ast_{k,\epsilon}(y,\lambda).
\end{equation}
The desired bounds \eqref{DCM9} follow from \eqref{DCM19}-\eqref{DCM21}, and Lemma \ref{mGk30} with $X=X_{L,\varrho_k}$, using also the bound 
\begin{equation}\label{DCM210}
\big\|\big[T_k^\ast(\lambda+i\epsilon)\big]^2R^\ast_{k,\epsilon}(\cdot,\lambda)\big\|_{X_{L,\varrho_k}}\lesssim \Big(\delta\wedge \frac{1}{|k|}\Big)^{-4}|k|^{5/2}\|\omega_{0k}\|_{H^3_k(I)}. 
\end{equation}

\end{proof}

\section{Proof of Theorem \ref{thm}}

In this section, we prove Theorem \ref{thm}. We can assume that $t\ge1$. We first give the proof of \eqref{th0.5}-\eqref{thm1}. Using the representation formula \eqref{F3.4}, we have
\begin{equation}\label{PMR1}
\begin{split}
\psi_k(t,y)&=\frac{1}{2\pi i}\lim_{\epsilon\to0+}\int_{\Sigma}e^{-ik\lambda t}\Big[\psi_{k,\epsilon}^+(y,\lambda)-\psi^-_{k,\epsilon}(y,\lambda)\Big]\,d\lambda\\
&=-\frac{1}{2\pi ik^2t^2}\lim_{\epsilon\to0+}\int_{\Sigma}e^{-ik\lambda t}\Big[\partial_\lambda^2\psi_{k,\epsilon}^+(y,\lambda)-\partial_\lambda^2\psi^-_{k,\epsilon}(y,\lambda)\Big]\,d\lambda.
\end{split}
\end{equation}
Fix $\Phi^\ast\in C_0^\infty(\Sigma_{\delta_0})$ with $\Phi^\ast\equiv 1$ on $\Sigma_{2\delta_0/3}$. We can decompose for $t\ge1, y\in [0,1]$,
\begin{equation}\label{PMR2}
\psi_k(t,y):=\psi^1_k(t,y)+\psi^2_k(t,y),
\end{equation}
where 
\begin{equation}\label{PMR3}
\begin{split}
\psi^1_k(t,y)&:=-\frac{1}{2\pi ik^2t^2}\lim_{\epsilon\to0+}\int_{\Sigma}e^{-ik\lambda t}(1-\Phi^\ast(\lambda))\Big[\partial_\lambda^2\psi_{k,\epsilon}^\iota(y,\lambda)-\partial_\lambda^2\psi^-_{k,\epsilon}(y,\lambda)\Big]\,d\lambda,\\
\psi^2_k(t,y)&:=-\frac{1}{2\pi ik^2t^2}\lim_{\epsilon\to0+}\int_{\Sigma}e^{-ik\lambda t}\Phi^\ast(\lambda)\Big[\partial_\lambda^2\psi_{k,\epsilon}^\iota(y,\lambda)-\partial_\lambda^2\psi^-_{k,\epsilon}(y,\lambda)\Big]\,d\lambda.
\end{split}
\end{equation}
For \eqref{th0.5}, it suffices to prove that for $\sigma\in\{1,2\}$, $k\in\Z\backslash\{0\}$ and $t\ge1$,
\begin{equation}\label{PMR4}
\big\|\psi^\sigma_k(t,\cdot)\big\|_{L^2([0,1])}\lesssim \frac{|k|^{3}}{t^2}\|\omega_{0k}\|_{H^3_k([0,1])}. 
\end{equation}
The case $\sigma=1$ in \eqref{PMR4} corresponding to the non-degenerate case is analogous to the case of monotonic shear flows, see \cite{JiaL}, and follow from Lemma \ref{bsdn1}-Lemma \ref{L1}. We focus on the main new case $\sigma=2$ in \eqref{PMR4}. Denote for $k\in\Z\backslash\{0\}$,
\begin{equation}\label{PMR4.5}
M_k:=|k|^{5/2}\|\omega_{0k}\|_{H^3_k([0,1])}.
\end{equation}
Our main tools are Lemmas \ref{DCM1}, Lemma \ref{DCM5} and Lemma \ref{DCM100}, which imply the following bounds for $y\in [0,1], \lambda\in \Sigma_{\delta_0}$.
\begin{itemize} 
\item If $|\lambda-b(y_\ast)|^{1/2}<|y-y_\ast|/20$, then
\begin{equation}\label{PMR5}
\begin{split}
&|\psi_k(y,\lambda)|\lesssim \big(\min\big\{|\lambda-b(y_\ast)|^{1/2}, |k|^{-1}\big\}\big)^{11/4}(|y-y_\ast|^{-1}+|k|)M_k,\\
&|\partial_\lambda^2\psi_k(y,\lambda)|\lesssim \big(\min\big\{|\lambda-b(y_\ast)|^{1/2}, |k|^{-1}\big\}\big)^{-5/4}(|y-y_\ast|^{-1}+|k|)M_k;
\end{split}
\end{equation}

\item If $|y-y_\ast|/20<|\lambda-b(y_\ast)|^{1/2}<20|y-y_\ast|$, then
\begin{equation}\label{PMR6}
\begin{split}
&|\psi_k(y,\lambda)|\lesssim \big(\min\big\{|\lambda-b(y_\ast)|^{1/2}, |k|^{-1}\big\}\big)^{5/4}|\lambda-b(y_\ast)|^{1/4}M_k,\\
&|\psi_k(y,\lambda)-\psi_k(y,b(y))|\lesssim |\lambda-b(y)|^{1/2}|\lambda-b(y_\ast)|^{3/8}M_k,\\
&\big\|\partial_\lambda^2\psi_k(\cdot,\lambda)-\Lambda_2(\cdot,\lambda)\big\|_{L^2(|y-y_\ast|\approx|\lambda-b(y_\ast)|^{1/2})}\lesssim (|\lambda-b(y_\ast)|^{-1/2}+|k|)^{9/4}|\lambda-b(y_\ast)|^{1/4}M_k;
\end{split}
\end{equation}

\item If $|\lambda-b(y_\ast)|^{1/2}>20|y-y_\ast|$, then
\begin{equation}\label{PMR7}
\begin{split}
&|\psi_k(y,\lambda)|\lesssim |\lambda-b(y_\ast)|^{1/4}\big(\min\big\{|\lambda-b(y_\ast)|^{1/2}, |k|^{-1}\big\}\big)^{5/4}M_k,\\
&\big\|\partial_\lambda^2\psi_k(\cdot,\lambda)-\Lambda_2(\cdot,\lambda)\big\|_{L^2(|y-y_\ast|<|\lambda-b(y_\ast)|^{1/2}/20)}\lesssim (|\lambda-b(y_\ast)|^{-1/2}+|k|)^{9/4}|\lambda-b(y_\ast)|^{1/4}M_k.
\end{split}
\end{equation}
\end{itemize}
It follows from \eqref{PMR5}-\eqref{PMR7} that for $y\in [0,1], t\ge1$,
\begin{equation}\label{PMR8}
\Big|\int_\R e^{-ik\lambda t}\Phi^\ast(\lambda)\Lambda_2(y,\lambda)d\lambda\Big|\lesssim |y-y_\ast|^{-1/4}\max\big\{1, |k|^{1/2}|y-y_\ast|^{1/2}\big\}M_k,
\end{equation}
and, by considering the cases $|\lambda-b(y_\ast)|\ll|y-y_\ast|^2$, $|\lambda-b(y_\ast)|\approx|y-y_\ast|^2$ and $|\lambda-b(y_\ast)|\gg|y-y_\ast|^2$, also that for $y\in [0,1], t\ge1$,
\begin{equation}\label{PMR9}
\Big\|\int_\R e^{-ik\lambda t}\Phi^\ast(\lambda)\big[\partial_\lambda^2\psi_k(y,\lambda)-\Lambda_2(y,\lambda)\big]d\lambda\Big\|_{L^2([0,1])}\lesssim |k|^{9/4}M_k.
\end{equation}
The desired bound \eqref{PMR4} for $\sigma=2$ follows from \eqref{PMR8}-\eqref{PMR9}. 

The proof of \eqref{thm1} is similar to the proof of \eqref{th0.5},  using Lemma \ref{DCM1} and Lemma \ref{DCM5}. 

We now turn to the proof of the depletion bounds \eqref{thm3}. Assume that $k\in\Z\backslash\{0\}$. Applying $-k^2+\partial_y^2$ to $\psi_k(t,y)$ in \eqref{F3.4}, and using \eqref{F7}, we get that for $y\in [0,1], t\ge1$,
\begin{equation}\label{PMR10}
\omega_k(t,y)=\omega_k^\ast(t,y)+\omega_{k}^{\ast\ast}(t,y),\end{equation}
where
\begin{equation}\label{PMR11}
\begin{split}
&\omega_k^\ast(t,y)\\
&:=\frac{1}{2\pi i}\lim_{\epsilon\to 0+}\int_{\Sigma}e^{-ik\lambda t}(1-\Phi^\ast(y))\bigg[\frac{b''(y)\psi_{k,\epsilon}^+(y,\lambda)-\omega_{0k}(y)}{b(y)-\lambda+i\epsilon}-\frac{b''(y)\psi_{k,\epsilon}^-(y,\lambda)-\omega_{0k}(y)}{b(y)-\lambda-i\epsilon}\bigg]\,d\lambda,\\
&\omega_k^{\ast\ast}(t,y):=\frac{1}{2\pi i}\lim_{\epsilon\to 0+}\int_{\Sigma}e^{-ik\lambda t}\Phi^\ast(y)\bigg[\frac{b''(y)\psi_{k,\epsilon}^+(y,\lambda)-\omega_{0k}(y)}{b(y)-\lambda+i\epsilon}-\frac{b''(y)\psi_{k,\epsilon}^-(y,\lambda)-\omega_{0k}(y)}{b(y)-\lambda-i\epsilon}\bigg]\,d\lambda.
\end{split}
\end{equation}
We have the bound for $t\ge1$,
\begin{equation}\label{PMR12}
\big\|\omega_k^\ast(t,y)\big\|_{L^\infty([0,1])}\lesssim |k|^2M_k.
\end{equation}
For $|y-y_\ast|<\delta_0/10, t\ge1$, since $(b(y)-\lambda+i\iota\epsilon)$ with $\iota\in\{\pm\}$ is not singular in this case, we have in addition by integration by parts that
\begin{equation}\label{PMR13}
|\omega_k^\ast(t,y)|\lesssim |k|^2\frac{1}{t}M_k. 
\end{equation}
We now turn to $\omega_k^{\ast\ast}(t,y)$. Using \eqref{BDC1}, we can write for $y\in[0,1], t\ge1$,
\begin{equation}\label{PMR14}
\begin{split}
&2\pi i\,\omega_k^{\ast\ast}(t,y)\\
&=\lim_{\epsilon\to 0+}\int_\R e^{-ik\lambda t}\Phi^\ast(\lambda)\bigg[\frac{\phi_{k,\epsilon}^+(y,\lambda)-(1-\Psi(y))\omega_{0k}(y)}{b(y)-\lambda+i\epsilon}-\frac{\phi_{k,\epsilon}^-(y,\lambda)-(1-\Psi(y))\omega_{0k}(y)}{b(y)-\lambda-i\epsilon}\bigg]\,d\lambda\\
&=\lim_{\epsilon\to 0+}\int_\R e^{-ik\lambda t}\Phi^\ast(\lambda)\bigg[\frac{\phi_{k,\epsilon}^+(y,\lambda)}{b(y)-\lambda+i\epsilon}-\frac{\phi_{k,\epsilon}^-(y,\lambda)}{b(y)-\lambda-i\epsilon}\bigg]\,d\lambda+W_k(t,y),
\end{split}
\end{equation}
where $W_k(t,y)$ satisfies the bound for $t\ge1$, 
\begin{equation}\label{PMR15}
\|W_k(t,\cdot)\|_{L^\infty([0,1])}\lesssim t^{-1}M_k,
\end{equation}
which follows from simple integration by parts argument. We decompose for $y\in [0,1]\backslash\{y_\ast\}$, 
\begin{equation}\label{PMR16}
\begin{split}
\omega_k^{\ast\ast}(t,y)-\frac{W_k(t,y)}{2\pi i}&=\frac{1}{2\pi i}\lim_{\epsilon\to 0+}\int_{\R}e^{-ik\lambda t}\Phi^\ast(\lambda)\bigg[\frac{\psi_{k}(y,\lambda)}{b(y)-\lambda+i\epsilon}\bigg]\,d\lambda\\
&+\frac{1}{2\pi i}\lim_{\epsilon\to 0+}\int_{\R}e^{-ik\lambda t}\Phi^\ast(\lambda)\phi^-_{k,\epsilon}(y,\lambda)\bigg[\frac{1}{b(y)-\lambda+i\epsilon}-\frac{1}{b(y)-\lambda-i\epsilon}\bigg]\,d\lambda.
\end{split}
\end{equation}
It follows from \eqref{PMR5}-\eqref{PMR7} that 
\begin{equation}\label{PMR17}
\begin{split}
&\bigg|\frac{1}{2\pi i}\lim_{\epsilon\to 0+}\int_{\R}e^{-ik\lambda t}\Phi^\ast(\lambda)\phi^-_{k,\epsilon}(y,\lambda)\bigg[\frac{1}{b(y)-\lambda+i\epsilon}-\frac{1}{b(y)-\lambda-i\epsilon}\bigg]\,d\lambda\bigg|\lesssim |y-y_\ast|^{7/4}M_k.\end{split}
\end{equation}
For $\gamma\in(1,\infty)$ to be fixed below, by considering the three ranges (I) $|\lambda-b(y_\ast)|\lesssim|y-y_\ast|^2$, (II) $|\lambda-b(y_\ast)|\ge\gamma|y-y_\ast|^2$, and (III) $|y-y_\ast|^2\ll|\lambda-b(y_\ast)|<\gamma|y-y_\ast|^2$, and using Lemma \ref{DCM5} and Lemma \ref{DCM100}, we get that 
\begin{equation}\label{PMR18}
\begin{split}
&\bigg|\frac{1}{2\pi i}\lim_{\epsilon\to 0+}\int_{\R}e^{-ik\lambda t}\Phi^\ast(y)\bigg[\frac{\psi_{k}(y,\lambda)}{b(y)-\lambda+i\epsilon}\bigg]\,d\lambda\bigg|\\
&\lesssim \Big[|y-y_\ast|^{7/4}\big(1+|k|^{1/2}|y-y_\ast|^{1/2}\big)+\frac{1}{|k|t}(|k|^{1/2}+\gamma^{-1/8}|y-y_\ast|^{-1/4})+\gamma^{7/8}|y-y_\ast|^{7/4}\Big]M_k.\end{split}
\end{equation}
In the above, we used integration by part to get decay in $t$ in range (II). Optimizing in $\gamma$, we get that for $t\ge1$,

(i) if $t|y-y_\ast|^2\lesssim1,$
\begin{equation}\label{PMR19}
\begin{split}
&\bigg|\frac{1}{2\pi i}\lim_{\epsilon\to 0+}\int_{\R}e^{-ik\lambda t}\Phi^\ast(y)\bigg[\frac{\psi_{k}(y,\lambda)}{b(y)-\lambda+i\epsilon}\bigg]\,d\lambda\bigg|\\
&\lesssim \Big[t^{-1}+|k|^{1/2}|y-y_\ast|^{7/4}+t^{-7/8}\Big]
\end{split}
\end{equation}
(ii) if $t|y-y_\ast|^2\gg1,$
\begin{equation}\label{PMR20}
\begin{split}
&\bigg|\frac{1}{2\pi i}\lim_{\epsilon\to 0+}\int_{\R}e^{-ik\lambda t}\Phi^\ast(y)\bigg[\frac{\psi_{k}(y,\lambda)}{b(y)-\lambda+i\epsilon}\bigg]\,d\lambda\bigg|\\
&\lesssim \Big[|y-y_\ast|^{7/4}\big(1+|k|^{1/2}|y-y_\ast|^{1/2}\big)+\frac{1}{|k|^{1/2}t^{7/8}}+|y-y_\ast|^{7/4}\Big]M_k.
\end{split}
\end{equation}
The desired bounds \eqref{PMR15}, \eqref{PMR17}, \eqref{PMR19}-\eqref{PMR20}. Theorem \ref{thm} is now proved.

\end{document}